\documentclass[a4paper]{amsart}
\usepackage{amsmath} 
\usepackage{amssymb}      
\usepackage{yhmath}% http://ctan.org/pkg/yhmath
\usepackage{mathdots}% http://ctan.org/pkg/mathdots
\usepackage{amsthm}      
\usepackage{tikz,amsmath}
\usetikzlibrary{arrows}

\newcounter{theoremintro}
\newtheorem{thmintro}[theoremintro]{Theorem}

\usepackage{epsfig}      
\usepackage{graphicx}
\usepackage[normalem]{ulem}
\usepackage[all,line,
]{xy} 
\CompileMatrices %goer at xy-matricerne koerer hurtigere.
\newcommand{\id}{\operatorname{id}} 
\newcommand{\Aut}{\operatorname{Aut}}

 \newcommand{\Ext}{\operatorname{Ext}}

\newcommand{\im}{\operatorname{im}}

\newcommand{\Span}{\operatorname{Span}}

\newcommand{\N}{{\mathbb{N}}}

\newcommand{\ev}{\operatorname{ev}}

   \theoremstyle{plain}%default
   \newtheorem{thm}{Theorem}[section]
   \newtheorem{prop}[thm]{Proposition}
   \newtheorem{lemma}[thm]{Lemma}  
   
   \theoremstyle{definition}
   
   \newtheorem{defn}[thm]{Definition}
   \newtheorem{example}[thm]{Example}
   \theoremstyle{remark}

\newenvironment{rcases}
  {\left.\begin{aligned}}
  {\end{aligned}\right\rbrace}

\usepackage{mdframed,xcolor}

\definecolor{mybgcolor}{gray}{0.8}
\definecolor{myframecolor}{rgb}{.647,.129,.149}

\mdfdefinestyle{mystyle}{
  usetwoside=false,
  skipabove=0.6em plus 0.8em minus 0.2em,
  skipbelow=0.6em plus 0.8em minus 0.2em,
  innerleftmargin=.25em,
  innerrightmargin=0.25em,
  innertopmargin=0.25em,
  innerbottommargin=0.25em,
  leftmargin=-.75em,
  rightmargin=-0em,
  topline=false,
  rightline=false,
  bottomline=false,
  leftline=false,
  backgroundcolor=mybgcolor,
  splittopskip=0.75em,
  splitbottomskip=0.25em,
  innerleftmargin=0.5em,
  leftline=true,
  linecolor=myframecolor,
  linewidth=0.25em,
}

\newmdenv[style=mystyle]{important}

\usepackage{lipsum}

   \numberwithin{equation}{section}

\title[Asymptotic lifting]{Asymptotic lifting for completely positive maps}

        \date{\today}
\date{\today}

\author{Marzieh Forough}
\address{Marzieh Forough
	Department of Mathematics, Faculty of Information Technology, Czech Technical University
	in Prague, Thakurova 9, 160 00, Prague 6, Czech Republic.
Institute of Mathematics, Czech Academy of Sciences
115 67
Praha 1, Czech republic}
\email{foroumar@fit.cvut.cz}

\author[Eusebio Gardella]{Eusebio Gardella}
\address{Eusebio Gardella
Department of Mathematical Sciences, Chalmers University of
Technology and University of Gothenburg, Gothenburg SE-412 96, Sweden.}
\email{gardella@chalmers.se}
\urladdr{www.math.chalmers.se/~gardella}

\author{Klaus Thomsen}
\address{Klaus Thomsen
Department of Mathematics, Aarhus University, Ny Munke\-gade,
8000 Aarhus C, Denmark}
\email{matkt@math.au.dk }

\thanks{
The first named author was supported by GA\v{C}R project 19-05271Y, RVO:67985840.
The second named author was partially supported
by the DFG through an \emph{eigene Stelle}
and under Germany's Excellence Strategy ``Mathematics 
M\"unster: Dynamics--Geometry--Structure''. The third named author was supported by the DFF-Research Project 2 `Automorphisms and Invariants of Operator Algebras', no.~ 7014-00145B}

\begin{document}

\begin{abstract} 
Let $A$ and $B$ be $C^*$-algebras with $A$ separable, let $I$ be an ideal in $B$, and let $\psi\colon A\to B/I$ be
a completely positive contractive
linear map. We show that 
there is a continuous family $\Theta_t\colon A\to B$, 
for $t\in [1,\infty)$, of lifts of $\psi$ that 
are asymptotically linear, asymptotically 
completely positive and asymptotically contractive. 
If $\psi$ 
is of order zero, then $\Theta_t$ can be chosen to have this property asymptotically. If $A$ and $B$ carry continuous actions of a second 
countable locally
compact group $G$ such that $I$ is $G$-invariant and 
$\psi$ is equivariant, we show that the family 
$\Theta_t$ can be chosen to be asymptotically equivariant.
If a linear completely positive lift for $\psi$ exists, we can arrange that $\Theta_t$ is linear and completely positive
for all $t\in [1,\infty)$. In the equivariant setting, if $A$, $B$ and $\psi$ are unital, we show that asymptotically linear unital lifts are only guaranteed to exist if $G$ is amenable. This 
leads to a new characterization of amenability in terms
of the existence of asymptotically equivariant unital
sections for quotient maps.
\end{abstract}

\maketitle

\renewcommand*{\thetheoremintro}{\Alph{theoremintro}}
\section{Introduction}

Lifts and sections for maps is a recurring theme in many mathematical fields. In the theory of $C^*$-algebras and its applications, one of the most prominent examples of this is the question about sections for a surjective $*$-homomorphism between $C^*$-algebras; for example in the work of Choi and Effros in \cite{CE} and in the theory of extensions of $C^*$-algebras, initiated in \cite{BDF} and \cite{Ar}. In view of the recent increased activity in research on $C^*$-dynamical systems, it seems inevitable that questions about lifts and sections will become more important than they already are in this category.
Addressing questions of this nature is the main motivation for the present work, while concrete applications will be given elsewhere; see \cite{FG}. 

Let $G$ be a second countable locally compact group. We shall work with pairs $(A,\alpha)$ where $A$ is a $C^*$-algebra and $\alpha\colon  G \to \Aut(A)$ is a homomorphism from $G$ into the group $\Aut(A)$ of automorphisms of $A$ (also called 
an action). When $\alpha$ is continuous in the sense that for all $a\in A$, the assignment $g \mapsto \alpha_g(a)$ is continuous as a map $G\to A$, we say that $(A,\alpha)$ is a \emph{G-algebra}. 
In many cases it is clear from the context what $\alpha$ is and we shall then use the notation
$g \cdot a= \alpha_g(a) $
for $g\in G$ and $a\in A$, since it often clarifies the statements. 

We are concerned with short exact sequences of $G$-algebras of the form
\begin{equation*}
\centerline{\begin{xymatrix}{
0 \ar[r] & (I, \gamma) \ar[r]^{\iota} & (A,\alpha) \ar[r]^-q  & (B,\beta)  \ar[r] & 0,}
\end{xymatrix}}
\end{equation*} 
where $\iota\colon I \to A$ and $q\colon A \to B$ are $G$-equivariant $*$-homo\-mor\-phisms such that $\iota$ is injective, $q$ is surjective and $\ker( q) = \im( \iota)$. Given another $G$-algebra $(S,\delta)$ and a map $\psi\colon S \to B$, which is not necessarily a $*$-homo\-mor\-phism, a \emph{lift} of $\psi$ is a map $\psi'\colon S \to A$ such that the diagram
\begin{equation*}
\begin{xymatrix}{
  &  A\ar[d]^-q \\
 S \ar[r]_-\psi \ar@{-->}^{\psi'}[ur]& B} 
\end{xymatrix}
\end{equation*} 
commutes. The most important case is when $\psi$ is linear and completely positive, and this is also the case we shall consider. In a setting where the group actions are absent (or trivial), the question about existence of a lift which is also linear and completely positive was considered by Choi and Effros in \cite{CE}.
They showed that a completely positive linear lift exists when 
the map $\psi$ is nuclear, which is in particular the case if 
either of $A$, $B$ or $S$ is nuclear. The main impetus for their work was the theory of extensions of $C^*$-algebras, which was also the main motivation for many subsequent examples, beginning with that of J. Anderson \cite{A}, showing that linear completely positive lifts do not always exist outside of the nuclear case. 

The present work was motivated by a study of $G$-algebras which made two of the authors suspect that if $\psi$ is $G$-equivariant and if a completely positive linear lift exists, then there also exists an \emph{almost equivariant}, completely positive linear lift $\psi'$, 
in the sense that given 
$\varepsilon >0$ and compact
subsets $K\subseteq G$ and $F\subseteq S$, we have 
$$
\max_{g\in K}\left\|\psi'(g \cdot s)- g \cdot \psi'(s)\right\| \leq \varepsilon
$$ 
for all $s\in F$. That this is indeed the case when $S$ is separable is a consequence of our Theorem~A below, which is inspired by the E-theory of Connes and Higson \cite{CH}. In fact,
the main result of this paper, Theorem \ref{04-01-21d}, deals with the same setting except that we neither assume that $\psi$ is equivariant nor that there is a linear completely positive contractive lift. We now reproduce part of its statement, and comment on it below.

\begin{thmintro}
Let $G$ be a second countable locally compact group, let $q\colon A\to B$ be a surjective equivariant map between $G$-algebras, let $S$ be a separable $G$-algebra, and let
$\psi\colon S \to B$ be a completely positive linear contraction. Then there is a continuous family $\Theta=(\Theta_t)_{t\in [1,\infty)} \colon S\to  A$ of lifts for $\psi$ which moreover is:
\begin{itemize}
 \item[(a)] asymptotically linear, asymptotically completely positive and asymptotically completely contractive (see Definition~\ref{15-01-21});
% for all $n\in\N$ and all $s\in M_n(S)$, we have 
% $$\lim\limits_{t\to \infty} \|(\Theta_t \otimes \id_{M_n})(s)\|  =  \|(\psi \otimes \id_{M_n})(s)\| .$$
% \item[(b)] asymptotically completely positive (see ??);
 %for $s,s'\in S$, we have 
 %$\lim_{t \to \infty} \|\Theta_t(s)\Theta_t(s')\| = \|\psi(s)\psi(s')\|$;
% \item[(d)] if $I$ is $\sigma$-unital, then
%$\lim\limits_{t \to \infty} \Theta_t(s)x   =  0$
%for all $s \in S$ and $x \in I$. 
 \item[(b)] asymptotically $(G,\psi)$-equivariant, meaning that for all $s,s'\in S$, we have 
 $$\lim_{t \to \infty} \left\| g \cdot \Theta_t(s) - \Theta_t(h\cdot s')\right\| = \left\|g \cdot \psi(s) - \psi(h\cdot s')\right\|,$$
 uniformly for $g,h$ in compact subsets of $G$.
 \item[(c)] For all $s,s'\in S$, we have
\[\lim_{t \to \infty} \left\|\Theta_t(s)\Theta_t(s')\right\| = \left\|\psi(s)\psi(s')\right\|.\]
\end{itemize}

Moreover, if a linear, completely positive lift for $\psi$ exists, then the continuous family $\Theta$ as above can be chosen so that, in addition, each map $\Theta_t$ is completely positive.

\end{thmintro}

Note that we do not assume $\psi$ to be equivariant in the theorem above.
In particular, condition (b) above implies that $\Theta$ is asymptotically equivariant whenever $\psi$ is equivariant, which is often the most interesting case. For some of the applications in \cite{FG}, however, the more general case is necessary. 

We now comment on the conclusion in part (c).
Following its introduction by Winter and Zacharias in \cite{WZ}, the notion of order zero for completely positive maps has become important in the classification of simple $C^*$-algebras. (Recall that a completely positive map $\psi\colon S\to B$ is said to be of \emph{order zero} if $\psi(s)\psi(s')=0$ whenever $ss'=0$.) In this setting, results that allow one to lift order zero maps from 
quotients are very relevant, and this often requires strong assumptions on the ideal (such as Kirchberg's notion of a 
$\sigma$-ideal; see Definition~1.5 in~\cite{Kir_central_06},
and see Definition~7.2 in~\cite{GarHirVac_strongly_2023} for 
its extension to the equivariant setting).
The conclusion in part~(c) of Theorem~A above implies that order zero maps 
admit asymptotically order zero lifts. We expect our results to be relevant in the study of the structure of C*-dynamical systems, particularly for actions of nonamenable groups on stably finite C*-algebras, which
has recently received increased attention; see \cite{GarGefKraNarVac_tracially_2024, GarKraVac_actions_2024, GarGefNarVac_dynamical_2024, GarLup_group_2021}.

When $S$, $A$ and $B$ are unital and $\psi$ is also unital, it is natural to look for asymptotically $(G,\psi)$-equivariant lifts that are also unital. These, however, will generally not exist even with the relaxed algebraic requirements described above. 
As it turns out, their existence characterizes amenability of $G$
(see Theorem \ref{05-01-21b}):

\begin{thmintro}
Let $G$ be a second countable locally compact group. Then 
the following are equivalent:
\begin{enumerate}
\item For every surjective equivariant map $q\colon A\to B$ between $G$-algebras, for every separable $G$-algebra $S$, and for 
every \emph{unital} completely positive linear contraction $\psi\colon S\to B$, there exists a continuous family $\Theta$ as in
the conclusion of Theorem~A which moreover satisfies $\Theta_t(1)=1$ for all $t \in [1,\infty)$.
\item $G$ is amenable.
\end{enumerate}
\end{thmintro}

Even without the $G$-actions, our Theorem~A provides new information. As pointed out above, there are by now a wealth of examples where the quotient map in an extension of separable $C^*$-algebras does not admit a completely positive and linear section. It follows from Theorem \ref{04-01-21d} below that there always exists a family of sections which has these properties asymptotically, and at the same time is asymptotically orthogonality preserving. This result should be compared with the examples in \cite{MT2} showing that there are separable $C^*$-algebras with extensions by the compact operators such that no other extension can be added to result in an extension for which the quotient map admits a family of sections that constitute an asymptotic $*$-homomorphism. Thus there are certainly limits to which properties of sections one can hope to get by relaxing algebraic conditions to asymptotic ones. Although 
it is conceivable that a linear lift can be proved to exist in the setting of Theorem~A, our results
appear to come close to being optimal.

Since this paper is almost exclusively devoted to the proof
of Theorem~A, which is quite long and technical, we describe 
the main conceptual steps in its proof. For the sake of the exposition, we focus on constructing a \emph{sequence} $(\Theta_n)_{n\in\N}$ as above. 
\vspace{.1cm}

\begin{enumerate}
\item[\textbf{Step 1:}] By taking (forced) 
unitizations everywhere, we may assume 
that $S$, $A$ and $\psi$ are unital, and that there exists
a $G$-invariant state $\chi$ on $S$. This is done
in Theorem~\ref{04-01-21d}.
\vspace{.1cm}

\item[\textbf{Step 2:}] Using standard arguments involving the Busby invariant of $0\to I \to A \to B\to 0$, we may assume that 
$A=M(I)$ is the multiplier algebra of $I$, that $B=Q(I)$ is
the associated Calkin algebra, that $S$ is a subalgebra of $B$
and that $\psi$ is the canonical inclusion. 
This is done in Proposition~\ref{18-01-21h}.
\vspace{.1cm}

\item[\textbf{Step 3:}] Let $(F_n)_{n\in\N}$ be an increasing sequence
of finite subsets of $S$ with dense union. Using the Bartle-Graves
theorem, one obtains the existence of a sequence $(\psi_n)_{n\in\N}$ of continuous lifts for $\psi$ such that 
\begin{enumerate}
 \item[(a)] $\psi_n$ is linear on the span of $F_n$;
 \item[(b)] $\psi_{n+1}$ agrees with $\psi_n$ on the span of $F_n$.
\end{enumerate}
This is done in Lemma~\ref{19-21-21b}. When a cp lift $\widetilde{\psi}$ for $\psi$ exists, then we may take 
$\psi_n$ to be $\widetilde{\psi}$ for all $n\in\N$. 
\vspace{.1cm}

\item[\textbf{Step 4:}] Using arguments of Kasparov \cite{K}, we carefully construct an almost $G$-invariant 
approximate identity $(y_n)_{n\in\N}$ in $I$, such that, when setting $\Delta_n=(y_n-y_{n-1})^{1/2}$, the sequence $(\psi^n)_{n\in\N}$ of maps $S\to M(I)$ given by
\[\psi^n(s)=\chi(s)y_n+\sum_{k=n}^\infty \Delta_k \psi_k(s)\Delta_k\]
is approximately $\mathbb{Q}[i]$-linear and approximately cpc \emph{on a dense subalgebra $S_0$ of $S$}, and approximately equivariant on 
$S_0$ with respect to a \emph{dense subgroup $G_0$ of $G$}. 
This is done in Lemma~\ref{01-02-21d}.

One would like to deduce, using 
continuity, that those properties then hold globally on $S$ and $G$. For this, one would need to know that for every $s\in S$,
the sequence $(\psi^n(s))_{n\in\N}$ is uniformly bounded. This 
is, however, not guaranteed in our construction, and fixing this is the 
goal of the next step.

\vspace{.1cm}

\item[\textbf{Step 5:}] Set 
\[\mathcal{A}=\big\{(a_n)_{n\in\N} \mbox{ bounded sequence in } A\colon a_n-a_1\in I \mbox{ for all } n\in\N\big\}\]
and let $\mathcal{I}$ denote the ideal of $\mathcal{A}$
consisting of all bounded sequences in $I$. For $n\in\N$, write 
$\pi_n\colon \mathcal{A}\to A$ for the projection onto the 
$n$-th coordinate.

The sequence 
$(\psi^n(s))_{n\in\N}$ obtained in the previous step gives us
a map $\Psi_0 \colon S_0 \to \mathcal{A}/\mathcal{I}$ which is
$\mathbb{Q}[i]$-linear, $G_0$-equivariant, and bounded. 
Thus, $\Psi_0$ 
extends by continuity to a $G$-equivariant linear map 
$\Psi\colon S\to \mathcal{A}/\mathcal{I}$ which is then
completely positive and contractive. 

Use Bartle-Graves again to find a continuous section $\sigma\colon \mathcal{A}/\mathcal{I}\to \mathcal{A}$, and set $\Theta_n=\pi_n\circ \sigma\circ \Psi\colon S\to A$ for all $n\in\N$.
Then $(\Theta_n)_{n\in\N}$ has the 
desired properties. This is done in detail in Theorem~\ref{02-02-21r}.
\end{enumerate}

\vspace{.2cm}

\textbf{Acknowledgements:} the authors would like to thank the 
anonymous referee for their very thorough feedback, which led to a
significant improvement of the presentation of the results.

\section{Lifting from Calkin algebras}\label{lifts}

In this section, we solve a particular case of the lifting problem
described in the introduction. More specifically, we will assume
that in the extension $0\to I \to A \to B \to 0$ we actually
have $A=M(I)$ and $B=M(I)/I$, so that we will focus on lifting
maps from the Calkin algebra $Q(I)=M(I)/I$ of a not necessarily unital
$C^*$-algebra $I$. In the following subsection, we focus on 
induced actions on Calkin algebras and aim at producing a suitable 
sequence $(\Delta_n)_{n\in\mathbb{N}}$ of positive contractions in
$M(I)$; see Lemma~\ref{31-01-21a}.
In the second subsection we will then use these elements to solve
the special case of the lifting problem described above.

\subsection{Actions on Calkin algebras}

%In this section we fix a locally compact second countable group $G$, a $G$-algebra $(U,\delta)$ such that $U$ is unital and separable, and a $G$-invariant state $\chi\colon U \to\mathbb{C}$; that is, a state on $U$ such that $\chi \circ \delta_g = \chi$ for all $g \in G$. We also fix a $G$-algebra $(I,\gamma)$ such that $I$ is $\sigma$-unital. 
For a locally compact group $G$ and a $C^*$-algebra $D$, we define an \emph{action} of $G$ on $D$ to be 
a group homomorphism $\delta\colon G\to \mathrm{Aut}(D)$.
Moreover, we say that $\delta$ is \emph{continuous} if for every $d\in D$, 
the map $G\to D$ given by $g \mapsto \delta_g(d)$ is continuous. 
In this case, we say that $(D,\delta)$ is a $G$-algebra.

It is common in the literature to drop the adjective ``continuous'' when discussing group actions,
and implicitly assume that all actions on $C^*$-algebras are continuous. There are, however, some relevant constructions in $C^*$-dynamics that produce non-continuous actions, even if one starts with 
continuous ones. Examples of such constructions are multiplier and Calkin algebras as well as ultraproducts and relative commutants (see \cite{GL} for the latter).
Since Calkin algebras play an important role in our work, we need
to introduce some terminology in order to deal with non-continuous 
actions.

\begin{defn}
Let $\delta \colon G\to \Aut(D)$ be a (not necessarily continuous) action of a locally compact group $G$ on 
a $C^*$-algebra $D$. We let
\[
D_{\delta}= \{d \in D \colon \mbox{the map } G\to D \mbox{ given by } g \mapsto \delta_g(d) \mbox{ is continuous}\}
\] 
denote the \emph{continuous part} of $D$ with respect to $\delta$. \end{defn}

It is not difficult to check that $D_\delta$ is a $C^*$-subalgebra of 
$D$ which is left invariant under $\delta$, and thus $\delta$ restricts
to a continuous action of $G$ on $D_{\delta}$. In other words, $(D_\delta,\delta)$ is a $G$-algebra.
The following is the main example of a not necessarily continuous action that 
we will consider in this work.

\begin{example}\label{eg:Calkin}
Let $G$ be a locally compact group and let $(I,\gamma)$ be a $G$-algebra. Then $\gamma$ extends to an action $\tilde{\gamma}$ of $G$ on the multiplier algebra $M(I)$ of $I$ satisfying $\tilde{\gamma}_g(m)x = \gamma_g(m\gamma_{g^{-1}}(x))$ for all $x \in I$ and all $g\in G$.
This action is in general not continuous (unless $G$ is discrete or $I$ is unital). Denote by $Q(I)=M(I)/I$ the associated (generalized) Calkin algebra and by 
$q_I\colon M(I)\to Q(I)$ the quotient map.
Since $I$ is $\tilde{\gamma}$-invariant, there is a (not necessarily continuous) action $\overline{\gamma}\colon G\to \Aut(Q(I))$ such that
$\overline{\gamma}_g\circ q_I=q_I\circ\tilde{\gamma}_g$ for all $g\in G$. Then $(M(I)_{\tilde{\gamma}},\tilde{\gamma})$
and $(Q(I)_{\overline{\gamma}},\overline{\gamma})$ are $G$-algebras,
and 
\[q_I\colon \big(M(I)_{\tilde{\gamma}},\tilde{\gamma}\big)\to \big(Q(I)_{\overline{\gamma}},\overline{\gamma}\big)\] 
is an equivariant
$\ast$-homomorphism.
\end{example}

The following lemma, which is a direct consequence of work by L.~Brown \cite{Br}, will be very useful to us.

\begin{lemma}\label{07-01-21} 
Let $G$ be a locally compact group and let $(I,\gamma)$ be a $G$-algebra. Then
\begin{equation*}
\centerline{\begin{xymatrix}{
0 \ar[r] & (I, \gamma) \ar[r] & (M(I)_{\tilde{\gamma}},\tilde{\gamma}) \ar[r]^-{q_I}  & (Q(I)_{\overline{\gamma}}, \overline{\gamma})  \ar[r] & 0,}
\end{xymatrix}}
\end{equation*} 
is a short exact sequence of $G$-algebras.
\end{lemma}
\begin{proof} The only non-trivial fact is that $ Q(I)_{\overline{\gamma}} \subseteq q_I\left(M(I)_{\tilde{\gamma}}\right)$, which follows from Theorem 2 in \cite{Br} since that theorem implies that the
restriction of $\tilde{\gamma}$ to $q_I^{-1}\left( Q(I)_{\overline{\gamma}}\right)$ is continuous.
\end{proof}

Our starting point is the following variation of the lemma on the top of page 152 in~\cite{K}. We include a proof because it is a key lemma and we shall rely on properties of the construction that are not explicit in \cite{K}.

\begin{lemma}\label{Kasparov} (Kasparov, \cite{K}) 
Let $G$ be a locally compact group and let $(I,\gamma)$ be a 
$\sigma$-unital
$G$-algebra.
Let $0 \leq d \leq 1$ be a strictly positive element in $I$ 
and let $M_0\subseteq M(I)_{\tilde{\gamma}}$ be a
separable $C^*$-subalgebra. Then there exists a countable approximate
unit $(x_n)_{n =1}^{\infty}$ for $I$ contained in $C^*(d)$ with the following properties:
\begin{itemize}
\item[(a)] $0 \leq x_n\leq 1$ for all $n\in \mathbb{N}$;
\item[(b)] $x_{n+1}x_n=x_n$ for all $n\in \mathbb{N}$;
\item[(c)] $\lim\limits_{n\to \infty}\|x_nb-bx_n\|=0$ for all $b\in M_0$;
\item[(d)] $\lim\limits_{n\to\infty} \max\limits_{g\in K}\|\gamma_g(x_n)-x_n\|=0$ for all 
compact subsets $K\subseteq G$.
\end{itemize}
\end{lemma}
\begin{proof} 
For each $n \in \mathbb N$, let $g_n\colon [0,1] \to [0,1]$ be the continuous function which vanishes between 0 and $\frac{1}{2}\left(\frac{1}{n} +  \frac{1}{n+1}\right)$; is linear between $\frac{1}{2}\left(\frac{1}{n} +  \frac{1}{n+1}\right)$ and $\frac{1}{n}$, and is constant equal to 1 for $t \geq \frac{1}{n}$. Using functional calculus, we set $a_n = g_n(d)\in I$. Then $0 \leq a_n \leq 1, a_{n+1}a_n = a_n$, and $\lim_{n \to \infty} a_n d= d$. In particular, $(a_n)_{n\in\N}$ is an approximate unit in $I$ since $\overline{dI} = I$. Since $G$ is $\sigma$-compact, the set 
$$
\big\{\gamma_g(m)\colon g \in G, m \in M_0 \cup \{d\}\big\}
$$ 
generates a separable $\gamma$-invariant $C^*$-subalgebra $M_{00}\subseteq M(I)_{\tilde{\gamma}}$ which contains $d$ and $M_0$. Let $(F_n)_{n=1}^{\infty}$ be an increasing sequence of finite subsets with dense union in $M_{00}$ and let $(V_n)_{n\in\N}$ be an increasing sequence of 
open subsets of $G$ with compact closures such that $\bigcup_{n\in\N} V_n = G$. For $N\in\N$, let $X_N$ denote the convex hull of $\{a_k \colon k \geq N\}$. 

We claim that 
%\begin{obs}\label{06-01-21c} 
given $n ,N\in \mathbb N$ and $\varepsilon > 0$, there is $x\in X_N$ such that $\|\gamma_g(x) - x\| \leq \varepsilon$ for all $g \in \overline{V_n}$ and $\|x b-bx\| \leq \varepsilon$ for all $b \in F_n$. To establish the claim, set $I_{00} = I \cap M_{00}$. For $x \in X_N$, let $f_x \in C(\overline{V_n},I_{00})$ and 
$h_x\colon F_n \to I_{00}$ be given by 
$$f_x(g) = \gamma_g(x) - x \ \ \mbox{ and } \ \ h_x(b) = xb-bx$$ 
for $g\in \overline{V_n}$ and $b \in F_n$.
Then $\Omega = \left\{ (f_x,h_x) \colon x \in X_N\right\}$ is a convex subset of the $C^*$-algebra $ C(\overline{V_n},I_{00}) \oplus C(F_n,I_{00})$. Assuming that the claim is not true, we can find $n,N\in \mathbb N$ such that the norm-closure of $\Omega$ does not contain $0$. By Hahn-Banach's separation theorem, there is then a $\delta > 0$ and elements $\Phi_1 \in C(\overline{V_n},I_{00})^*$ and $\Phi_2 \in C(F_n,I_{00})^*$ such that 
\begin{equation*}
|\Phi_1( f_x) + \Phi_2(h_x)|  \geq  \delta
\end{equation*} 
for all $x \in X_N$. Write $\Phi_1 = \omega_1 - \omega_2 + i (\omega_3 - \omega_4)$, where the $\omega_j$'s are positive linear functionals on $ C(\overline{V_n},I_{00})$.
For each $k\geq N$ define $A_k,B_k \in C\left(\overline{V_n},I_{00}\right)$ by $A_k(g) = a_k$ and $B_k(g) = \gamma_g(a_k)$ for $g \in \overline{V_n}$. Then $(A_k)_{k\in\N}$ and $(B_k)_{k\in\N}$ are both approximate units in $C(\overline{V_n},I_{00})$ and hence $\left\|\omega_j\right\| = \lim_{k \to \infty} \omega_j(B_k) = \lim_{k \to \infty} \omega_j(A_k) $ for all $j\in\N$. This implies that 
\[\lim_{k\to \infty} \Phi_1\left(B_k - A_k\right) = 0.\] Since $B_k -A_k = f_{a_k}$ it follows that there is an $L \geq N$ such that
\begin{equation}\label{06-01-21b}
\left|\Phi_2(h_x)\right|  \geq  \frac{\delta}{2} 
\end{equation}
for all $x\in X_L$. Let $C_x \in C(F_n,I_{00})$ and $B \in C(F_n, M_{00})$ be the elements defined such that $C_x(b) = x$ and $B(b) = b$ for $b \in F_n$. Then $(C_x)_{x \in X_L}$ is a convex approximate unit in $C(F_n,I_{00})$ which is an ideal in $C(F_n,M_{00})$. Since $h_{x} = C_xB-BC_x$, the inequality \eqref{06-01-21b} contradicts the lemma on page 330-331 in \cite{Ar}. This proves the claim.

It is now straightforward to use the claim 
to construct an increasing sequence $(k_n)_{n\in\N}$ in 
$\mathbb N,$ and elements $x_n$ is the convex hull of 
$$
\left\{ a_j\colon \ k_n \leq j \leq k_{n+1}-1\right\}
$$ 
such that
\begin{itemize}
\item $\|x_nb-bx_n\|\leq \frac{1}{n}$ for all $b\in F_n$;
\item $\|\gamma_g(x_n)-x_n\|\leq \frac{1}{n}$ for all $g \in V_n$, and
\item $\left\|x_n d - d \right\| \leq \frac{1}{n}$.
\end{itemize}
The sequence $(x_n)_{n\in\N}$ will then have the desired properties.
\end{proof}

The following is the main technical lemma of this subsection, and 
will be needed in the upcoming one in order to produce suitable
lifts along the quotient map $M(I)\to Q(I)$.

\begin{lemma}\label{31-01-21a} 
Let $G$ be a locally compact group, let $(I,\gamma)$ be a 
$\sigma$-unital $G$-algebra,
and let $d\in I$ be a strictly positive element. 
Let $(Z_n)_{n\in\N}$ be an increasing sequence of compact
subsets of $M(I)_{\tilde{\gamma}}$, and let $(V_n)_{n\in\N}$
be an increasing sequence of open subsets of $G$ with compact
closures.
Then there is an approximate unit $(y_n)_{n=1}^{\infty}$ for $I$ contained in $C^*(d)$, such that, if we set
\[
\Delta_0 = \sqrt{y_{1}} \ \text{and} \  \Delta_n = \sqrt{y_{n+1} - y_n},
\]
for $n\geq 1$, then the following properties are satisfied:
\begin{itemize}
\item[(i)] $0 \leq y_n  = y_{n+1}y_n\leq y_{n+1} \leq 1$ for all $n\in\N$,
\item[(ii)] $\lim\limits_{n \to \infty}\max\limits_{g\in K} \|\gamma_g(y_n) -y_n\| = 0$ for all compact subsets $K\subseteq G$,
\item[(iii)] $\left\|z(1-y_n)\right\|\leq\left\|q_I(z)\right\| + \frac{1}{n}$ for all $z \in Z_n$, 
%\end{itemize}
%and there is an increasing sequence $i_1 < i_2 < i_3 < \cdots$ in $\mathbb N$ such that when we s
%the sequence $\left\{\Delta_n\right\}_{n=0}^\infty$ has the following properties: 
%\begin{itemize}
\item[(iv)] $\Delta_{n+1} y_n = 0$,
\item[(v)] $\left\| \Delta_n z -z \Delta_n \right\|\leq2^{-n}$ for all $z \in Z_n$,
\item[(vi)] $\left\|\gamma_g(\Delta_n) - \Delta_n\right\|\leq2^{-n}$ for all $g \in V_n$,
\end{itemize}
 for all $n\geq 1$.
\end{lemma}
\begin{proof} Let $M_0$ be the $C^*$-subalgebra of $M(I)_{\tilde{\gamma}}$ generated by $\bigcup_{n=1}^\infty Z_n$, and let $(x_n)_{n=1}^\infty$ be an approximate unit in $I$ contained  in $C^*(d)$ satisfying properties (a), (b), (c) and (d) in Lemma~\ref{Kasparov}. It is easy to see, by approximating the square root function on $[0,1]$ by polynomials, that for every $\varepsilon > 0$ there is $\delta > 0$ such that when $a \in I$ is a positive contraction 
satisfying $\left\|\gamma_g(a) - a\right\| \leq \delta$, then $\left\|\gamma_g(\sqrt{a}) - \sqrt{a}\right\| \leq \varepsilon$. 
Combined with the lemma on page 332 in \cite{Ar}, it follows that for each $n \in \mathbb N$ there is a $\delta_n > 0$ such that whenever $a \in I$ is a positive contraction, then
\begin{equation*}
\begin{rcases} 
\left\|\gamma_g(a) - a\right\| \leq \delta_n  \ \forall g \in V_n\\
\left\|az-za\right\| \leq \delta_n \ \forall z \in Z_n 
\end{rcases}
\Rightarrow
\begin{cases}
\left\|\gamma_g(\sqrt{a}) - \sqrt{a}\right\| \leq 2^{-n} \ \forall g \in V_n\\
\left\|\sqrt{a}z-z\sqrt{a}\right\| \leq 2^{-n} \ \forall z \in Z_n \   .
\end{cases}
\end{equation*}
We extract therefore a subsequence $(y_n)_{n\in\N}$ 
from $(x_n)_{n\in\N}$ such that 
$$
\left\|\gamma_g(y_{n+1} -y_n) - (y_{n+1} - y_n)\right\| \leq \delta_n
$$ 
for all $g \in V_n$ and $\left\|(y_{n+1} -y_n) z -z(y_{n+1} -y_n)\right\| \leq \delta_n$ for all $z \in Z_n$. Since $\lim_{n \to \infty} \left\|m (1-x_n)\right\| = \left\|q_I(m)\right\|$ for all $m \in M(I)$, we can also arrange that (iii) holds. Then $(y_n)_{n=1}^{\infty}$ will have all the stated properties.
\end{proof}

The following is Lemma 3.1 in \cite{MT1}, and we reproduce it here 
for the convenience of the reader, as we will use it repeatedly.

\begin{lemma}\label{01-02-21c}
Adopt all the assumptions and notations from Lemma~\ref{31-01-21a}, and 
let $(y_n)_{n=1}^{\infty}$ and $(\Delta_n)_{n=0}^\infty$ be as in its conclusion. 
Let $(m_n)_{n =0}^\infty$ be a uniformly bounded sequence in $M(I)$. Then the sum
\[
\sum_{n=0}^{\infty} \Delta_n m_n \Delta_n
\]
converges in the strict topology of $M(I)$ and for $k\in\N$
we have 
\[
\bigg\| \sum_{n=k}^{\infty} \Delta_n m_n \Delta_n\bigg\|\leq\sup_{n\geq k} \|m_n\|.\]
\end{lemma}

%{\color{red}
%I would add the following to the statement:
%``In particular, if $\lim_{n\to\infty}\|m_n\|=0$, then the 
%series converges in norm to an element of $I$.''
%that if $m_n$ belongs to 
%$M(I)_{\tilde{\gamma}}$ for all $n\in\N$, then the series
%also belongs to $M(I)_{\tilde{\gamma}}$. (It seems to me, although
%I may be wrong, that this is used later
%}

\subsection{Producing lifts into multiplier algebras}

As the main step towards more general cases, we consider first a
unital linear completely positive map $\psi\colon U \to Q(I)_{\overline{\gamma}}$, and we seek to construct a lift $\psi'$ in the diagram
\begin{equation*}
\begin{xymatrix}{
  &  M(I)_{\tilde{\gamma}}\ar[d]^-{q_I} \\
 U \ar[r]_-\psi \ar@{-->}^-{\psi'}[ur]& Q(I)_{\overline{\gamma}} } 
\end{xymatrix}
\end{equation*} 
such that $\psi'$ has properties as close as possible to those of $\psi$. (This is rather vague at this point, but it will soon become 
clear; see Definition~\ref{25-02-21a}.)

While the focus in the following is on the general case, we want to simultaneously handle the case where there exists a unital completely positive linear lift $\psi'$, in which case
the problem reduces to that of making adjustments to $\psi'$ so that it respects the $G$-actions as much as possible. As will become clear, the arguments necessary to handle the latter situation are much simpler than the ones we present to handle the general case, but they also lead to stronger conclusions. We will treat both situations in parallel.

 We shall work with maps between $C^*$-algebras that may not respect anything of the algebraic structure. For this reason, we will be very explicit about the properties of the maps we consider.  

\begin{defn}\label{19-01-21c} Let $\theta\colon A \to B$ be a (not necessarily continuous or linear) 
map between $C^*$-algebras. We say that $\theta$ is
\begin{itemize}
\item \emph{self-adjoint}, if $\theta(a^*) = \theta(a)^*$ for all $a \in A$.
\item \emph{unital}, if $A$ and $B$ are unital and $\theta(1) = 1$.
%\item $\theta$ is a contraction when $\left\|\theta(a)\right\|\leq\|a\|$ for all $a \in A$.
%positive when $a \geq 0 \ \Rightarrow \ \psi(a) \geq 0$
\end{itemize}
Let $n \in \mathbb N$. We denote by $\theta \otimes \id_{M_n}$ the map $\theta\otimes\id_{M_n}\colon M_n(A) \to M_n(B)$ given by entry-wise application of $\theta$. 
\end{defn}

The following lemma, which is essentially an application of the 
Bartle-Graves selection theorem (stating that every bounded, surjective map between Banach spaces admits a not-necessarily linear, continuous section), will allow us to consistently construct solutions to our 
lifting problem that are linear on larger and larger finite-dimensional subspaces. 

\begin{lemma}\label{19-21-21b} Let $U$ and $E$ be $C^*$-algebras, with $U$ unital, and let $L\colon D \to Q(E)$ be a unital, continuous, self-adjoint, linear map.
\begin{itemize}
\item[(i)] For each finite set $F \subseteq U$ there is a unital continuous and self-adjoint map $L_F\colon U\to M(E)$ which is linear on $\Span(F)$ and satisfies $q_E \circ L_F = L$.
\item[(ii)] Let $F_1\subseteq F_2$ be finite subsets of $U$, and 
let $L_1\colon U\to M(E)$ be a unital, continuous self-adjoint map 
which is linear on $\Span(F_1)$ and satisfies $q_E \circ L_1 = L$.
There is a unital continuous and self-adjoint map $L_2\colon U\to M(E)$, which is linear on $\Span(F_2)$ and satisfies $q_E \circ L_2 = L$ and $L_2(x) = L_1(x)$ for $x \in \Span(F_1)$.
\end{itemize}
\end{lemma}
\begin{proof} (i) Set $F' = \{1\} \cup F \cup F^*$. By the Bartle-Graves selection theorem, \cite{BG}, there is a continuous map $s_0\colon Q(E) \to M(E)$ such that $q_E \circ s_0=\mathrm{id}_{Q(E)}$. By exchanging $s_0$ with $s_0 - s_0(0)$ we may assume that $s_0(0) = 0$. Consider the Banach space quotient $Q(E)/L(\Span(F'))$ and the corresponding quotient map $\pi\colon Q(E) \to Q(E)/L(\Span(F'))$. The Bartle-Graves selection theorem gives a continuous map 
$$s_1\colon Q(E)/L(\Span(F')) \to Q(E)$$ 
such that $\pi \circ s_1$ is the identity on $Q(E)/L(\Span(F'))$, and again we may assume that $s_1(0) = 0$. Since $\Span L(F')$ is finite dimensional and $1 \in L(F')$, there is a continuous and linear map 
$$s_F\colon  L(\Span(F')) \to M(E)$$ 
with $s_F(1) = 1$ such that $q_E \circ s_F=\id_{L(\Span(F'))}$. Since $x - s_1(\pi(x)) \in L(\Span(F'))$ for all $x \in Q(E)$, we can define $\theta\colon Q(E) \to M(E)$ by
$$
\theta(x)=s_0\left(s_1(\pi(x))\right)+s_F( x - s_1(\pi(x))).
$$
Then $\theta$ is continuous, linear on $L(\Span(F'))$, and 
satisfies $\theta(1) = 1$ and $q_E \circ \theta=\id_{Q(E)}$.
By replacing $\theta(x)$ with $\frac{1}{2}(\theta(x)+\theta(x^*)^*)$ for $x\in Q(E)$, we may also assume that
$\theta$ is self-adjoint. Then $L_F = \theta \circ L$ has 
the desired properties.

(ii) Apply first (i) to get a continuous and self-adjoint unital map $L'_2\colon U\to M(E)$ which is linear on $\Span(F_2 \cup F_2^*)$ and satisfies $q_E \circ L'_2= L$. 
Since $\Span(F_1 \cup F_1^*)$ is finite-dimensional, there is a continuous linear projection $P'\colon U \to \Span(F_1 \cup F_1^*)$. Set
$$
P(u)=\frac{1}{2}\left(P'(u) + P'(u^*)^*\right)
$$ 
for all $u\in U$.
Then the map $L_2\colon U\to Q(E)$ given by
$$
L_2(u)=L_1(P(u))+L'_2(u-P(u)),
$$
for all $u\in U$,
has the desired properties.
\end{proof}

The following is the main technical lemma of this section, and, roughly speaking, asserts that the lifting problem we are interested in can 
be solved on dense subsets of $U$ and $G$.

\begin{lemma}\label{01-02-21d} 
Let $G$ be a locally compact, second countable 
group, let $(I,\gamma)$ be a $\sigma$-unital $G$-algebra,
and let $d\in I$ be a strictly positive element. Denote by 
$q_I\colon M(I)\to Q(I)$ the canonical quotient map.
Let $(U,\delta)$ be a separable, unital $G$-algebra, let $\chi\colon U\to \mathbb{C}$ be a $G$-invariant state, and let $\psi\colon U\to Q(I)_{\overline{\gamma}}$ be a unital, completely positive linear map.

Let $G^{(0)}$ be a countable dense subset of $G$,
and let $(F_n)_{n\in\mathbb{N}}$ be an increasing sequence 
of self-adjoint 
finite subsets of $U$ whose union $U^{(0)}$ is dense in $U$ 
and satisfies 
\begin{equation}\label{04-01-21g}
\Span_{\mathbb Q[i]}(U^{(0)})\subseteq U^{(0)}, \  \ 
G^{(0)} \cdot U^{(0)} \subseteq U^{(0)} \ \mbox{ and } \ u -\chi(u) \in U^{(0)}
\end{equation}
for all $u \in U^{(0)}$. (This choice is always possible\footnote{For example, let $(F_n')_{n\in\N}$ be an increasing sequence of finite subsets
of $U$ with dense union, and assume that $F_1'$ contains $\{0,1\}$. 
Let $(Q_n)_{n\in\N}$ be an increasing sequence of finite subsets of $\mathbb Q[i] $ with $\bigcup_{n=1}^\infty Q_n=\mathbb Q[i]$ and
such that $\{0,1\} \subseteq Q_1$. Let $(G_n)_{n\in\N}$ be an increasing sequence of finite subsets
of $G$, all containing the unit, with $G^{(0)} = \bigcup_n G_n$. 
Set $F_1 = F_1'$ and define $F_n$ for $n \geq 2$ recursively such that $F_n = F''_n \cup {F''_n}^* \cup F'_n \cup {F'_n}^*$, where
\begin{align*}
&F''_n = \Big\{\sum_{j=1}^n g_j \cdot (q_j u_j)\colon g_j \in G_n, q_j \in Q_n, u_j \in F_{n-1}\Big\} \cup \left\{ u - \chi(u) \colon \ u \in F_{n-1}\right\}.
\end{align*}
Then $(F_n)_{n\in\mathbb{N}}$ satisfies the desired properties.}.)

Then there exists a family $\psi^t\colon U \to M(I)$, for $t \in [1,\infty)$, of self-adjoint maps with the following properties:
\begin{enumerate}
 \item[(1)] $\psi^t(1) = 1$ for all $t \geq 1$;
\item[(2)] The assignment $t \mapsto \psi^t(u)$ is continuous for all $u \in U$;
\item[(3)] $\psi^t$ is continuous for all $t \geq 1$;
\item[(4)] $\psi^t$ is linear on $\Span(F_n)$ when $t \geq n+1$;
\item[(5)] $q_I \circ \psi^t = \psi$ for all $t \geq 1$;
\item[(6)] $\psi^t(U) \subseteq M(I)_{\tilde{\gamma}}$ for all $t\geq 1$.
\item[(7)] Let $ u,v \in U^{(0)}$, $g,h\in G^{(0)}$, and let $m\in\{1,d\}$. Then
\begin{align*}
&\limsup_{t \to \infty} \left\|m\left( g \cdot \psi^t(u)-\psi^t(h \cdot v)\right) \right\|\\
& \ \ \ \ \ \ \ \ \ \ 
\leq\max\left\{\left|\chi(u-v)\right|, \left\|q_I(m) \left(g \cdot\psi( u)  - \psi(h\cdot v)\right)\right\|\right\}.
\end{align*}
\item[(8)] Let $u,v \in U^{(0)}$. Then
\begin{align*}
& \limsup_{t \to \infty} \left\|\psi^t(u -\chi(u))\psi^t(v -\chi(v))\right\|\\
& \ \ \ \ \ \ \ \ \ \ \leq \left\|\psi(u - \chi(u))\psi(v - \chi(v))\right\|.
\end{align*}
 \item[(9)] For $j \in \mathbb N$, write $M_j(U^{(0)})$ for 
the subset of $M_j(U)$ of matrices with entries in $U^{(0)}$. 
For $u\in M_j(U^{(0)})$, we have 
$$
\limsup_{t \to \infty} \left\|(\psi^t \otimes \id_{M_j})(u)\right\| \leq \max\left\{\left\|(\chi \otimes \id_{M_j})(u)\right\|, \left\|(\psi\otimes \id_{M_j})(u)\right\| \right\}.
$$
 \end{enumerate}
Moreover, if a unital and completely positive lift for $\psi$ exists, then we can choose the maps $\psi^t$ as above to additionally be linear, completely positive and contractive.
\end{lemma}

\begin{proof}
We first present the construction of the maps $\psi^t$, and later verify that they satisfy the properties claimed above.

By part~(i) in Lemma~\ref{19-21-21b}, there is a unital continuous and self-adjoint map $\psi_1\colon U \to M(I)$ such that $q_I \circ \psi_1 = \psi$ and $\psi_1$ is linear on $\Span(F_1)$.
By repeatedly using of part~(ii) in Lemma~\ref{19-21-21b}, we get unital continuous self-adjoint maps $\psi_n\colon U \to M(I)$ satisfying the following for all $n\in\N$:
\begin{itemize}
\item[(a)] $q_I \circ \psi_n=\psi$
\item[(b)] $\psi_n$ is linear on $\Span(F_n)$, and
\item[(c)] $\psi_{n+1}(x) = \psi_n(x)$ for all $x \in F_n$.
\end{itemize} 
It follows from (a) and Lemma~\ref{07-01-21} that
\begin{itemize}
\item[(d)] $\psi_{n}(U) \subseteq M(I)_{\tilde{\gamma}}$ for all $n\in\N$.
\end{itemize} 
In the case where a unital completely positive linear lift $\widetilde{\psi}$ of $\psi$ is given, we take all the $\psi_n$'s above to be equal to $\widetilde{\psi}$.

Let $e\in V_1 \subseteq V_2 \subseteq V_3 \subseteq \cdots$ be open subsets of $G$ with compact closures $\overline{V_n}$ such that $\bigcup_{n\in\mathbb{N}} V_n = G$. For each $n\in \N$, set 
\[Y_n=\bigcup_{j=1}^n \big\{\tilde{\gamma}_g(\psi_j(u)) - \psi_j(\delta_h(v))\in M(I)_{\tilde{\gamma}}\colon g,h \in \overline{V_n}, u,v \in F_n\big\}\subseteq M(I)_{\tilde{\gamma}},
\]
and let $Z_n \subseteq M(I)_{\tilde{\gamma}}$ be a compact set containing the set
\[
Y_n\cup dY_n \cup \bigcup_{j=1}^n \left(\psi_j(F_n)\psi_j(F_n) \cup \psi_j(F_n)\right).\]

We fix from now on an approximate unit $(y_n)_{n\in\N}$ for $I$
contained in $C^*(d)$ satisfying the conclusion of 
Lemma~\ref{31-01-21a} for the sets $(V_n)_{n\in\N}$ and $(Z_n)_{n\in\N}$, and set $\Delta_0=\sqrt{y_1}$ and 
$\Delta_n=\sqrt{y_{n+1}-y_n}$ for $n\geq 1$.

For $k\in\N$ and using Lemma~\ref{01-02-21c}, we define a map $\psi^k\colon U\to M(I)$ by
\begin{equation}\label{02-02-21j}
 \psi^k(u)= \chi(u)y_k+ \sum_{n=k}^\infty \Delta_n\psi_k(u)\Delta_n
\end{equation}
for all $u\in U$. Moreover, for $t\in [k,k+1]$, 
we define the map $\psi^t$ to be  
\begin{equation}\label{02-02-21k}
\psi^t = (t-k+1)\psi^{k} + (k-t)\psi^{k-1}.
\end{equation}
We thus obtain a family
$\psi^t\colon U\to M(I)$, for $t\geq 1$. 
%(Although we will not 
%need this here, the map $\psi^t$ can be explicitly defined as
%$\psi^t(u)= \sum_{n=0}^\infty \Delta_n \psi^t_n(u) \Delta_n$
%for all $u \in U$.) 
We now check that the maps $\psi^t$ satisfy the properties in the 
statement in the lemma.

(1) This follows immediately from \eqref{02-02-21j} and \eqref{02-02-21k}, since $\chi$ and the $\psi_n$'s are all unital.

(2) This follows directly from \eqref{02-02-21k}.

(3) By \eqref{02-02-21k}, it suffices to show that
$\psi^k$ is continuous for each fixed $k\in\N$. 
Moreover, by \eqref{02-02-21j} it suffices to show that 
the assignment $u\mapsto \sum\limits_{n=k}^\infty \Delta_n\psi_k(u)\Delta_n$ is continuous. 
Let $u \in U$ and $\varepsilon > 0$ be given. 
Using continuity of $\psi_k$, find
$\alpha > 0$ such that 
$\|\psi_k(u) - \psi_k(v)\|\leq \varepsilon$ when $\|u -v \| \leq \alpha$. It thus follows from Lemma~\ref{01-02-21c} 
that 
$$\Big\|\sum\limits_{n=k}^\infty \Delta_n(\psi_k(u)-\psi_k(v))\Delta_n\Big\|\leq \varepsilon,$$ 
when $\|u-v\|\leq \alpha$, as desired.

(4) This follows from \eqref{02-02-21k}, since $\psi_k$, and thus also $\psi^k$ by \eqref{02-02-21j}, is linear on $\Span(F_n)$ when $k \geq n$. 

(5) By \eqref{02-02-21k}, it suffices to show that $q_I \circ \psi^k(u) = \psi(u)$ for all $k\in\N$. Fix $k\in \N$. 
Since $U^{(0)}$ is dense in $U$ and $\psi^k$ is 
continuous by part (3), it suffices to show that for every $r\in\N$ and every $u\in F_r$, we have
$q_I \circ \psi^k(u) = \psi(u)$. 
Fix $r\in\N$. Without loss of generality, we assume that $r\geq k$.

We need some auxuliary maps,
which we define next.
For each $n \in \mathbb N \cup \{0\}$, set
\[
\psi^k_n(u)=\begin{cases}  
\chi(u)  , & n  \leq k - 1\\
\psi_{k}(u) , &  n \geq k ,
\end{cases}\]
for all $u\in U$. 
%For $t \in [k-1,k]$ and $u\in U$, note that the element $\psi^t_n(u)$ belongs to the 
%convex hull of the set $\left\{\chi(u),\psi_{k-1}(u), \psi_k(u)\right\}$.
%Thus,
%$$
%\big\|\psi^t_n(u)\big\|\leq \max \left\{\left|\chi(u)\right|, \left\|\psi_{k-1}(u)\right\|, \left\|\psi_k(u)\right\|\right\} \ 
%$$
%for all $n\in\N$. 
Then
\begin{equation}\label{02-02-21l}
\psi^k(u) = \sum_{n=0}^{r-1} \Delta_n\psi^k_n(u)\Delta_n +\sum_{n=r}^\infty \Delta_n \psi_k(u)\Delta_n
\end{equation}
for all $u\in U$.
It follows from (v) in Lemma~\ref{31-01-21a} and the 
fact that $\|\Delta_n\|\leq 1$ that 
$$
\left\|\Delta_n^2\psi_k(u) - \Delta_n\psi_k(u)\Delta_n\right\|\leq \left\|\Delta_n\psi_k(u) - \psi_k(u)\Delta_n\right\| \leq 2^{-n} 
$$
for all $u\in F_r$, 
when $n \geq k$. Using that the series $\sum_{n=r}^{\infty}\Delta_n^2$ converges in the strict topology of $M(I)$ to $1-y_r$ at the first step, we therefore have
\begin{align*}
(1-y_{r})\psi_k(u) - \sum_{n=r}^\infty \Delta_n \psi_k(u)\Delta_n &= \sum_{n=r}^{\infty}\Delta_n^2\psi_k(u) - \sum_{n=r}^\infty \Delta_n \psi_k(u)\Delta_n   \\
&=\sum_{n=r}^\infty \left[\Delta_n^2\psi_k(u) - \Delta_n\psi_k(u)\Delta_n\right], 
\end{align*}
where the last sum converges in norm by the previous norm estimate.
In particular, since $\Delta_n^2\psi_k(u) - \Delta_n\psi_k(u)\Delta_n$ 
belongs to $I$ (because $\Delta_n\in I$), it follows that
\[(1-y_{r})\psi_k(u) - \sum_{n=r}^\infty \Delta_n \psi_k(u)\Delta_n
 \in I.
\]
We combine the above with \eqref{02-02-21l}, writing $=_I$
for equality modulo $I$:
\begin{align*}
\psi^k(u)-\psi_k(u) 
&= \psi^k(u)-(\psi_k(u)-y_r\psi_k(u))-y_r\psi_k(u)\\
&=_I \psi^k(u) - \sum_{n=r}^\infty \Delta_n \psi_k(u)\Delta_n\\
&= \sum_{n=0}^{r-1} \Delta_n\psi^k_n(u)\Delta_n =_I 0.
\end{align*}
It follows that $q_I( \psi^k(u)) = q_I (\psi_k(u)) = \psi(u)$, as desired. 

(6) This follows from part (5) and Lemma~\ref{07-01-21}, since $\psi(U) \subseteq Q(I)_{\overline{\gamma}}$ by assumption.

(7) By \eqref{02-02-21k}, it suffices to show that
\begin{align}\label{part7}
\begin{split}
&\limsup_{k\to\infty}\|m( g \cdot \psi^k(u)-\psi^k(h \cdot v)) 
\|\\
&  \ \ \ \ \ \ \ \ \ \
\leq\max\big\{|\chi(u-v)|, \|q_I(m) (g \cdot\psi( u)  - \psi(h\cdot v))\|\big\}.
\end{split}
\end{align}
Let $n\in\N$ such that $u,v\in F_n$ and $g,h\in V_n$, 
and fix $k\geq n+1$.  We  write
$$
g \cdot \psi^k (u)-\psi^k(h \cdot v)=a+b ,
$$
where 
$$
 a  =  \sum_{j=0}^{k-1} \big[{\gamma}_g(\Delta_j)\chi(u){\gamma}_g(\Delta_j)  -  \Delta_j\chi(v)\Delta_j\big]
$$
and
$$
b  =   \sum_{j=k}^{\infty} \big[{\gamma}_g(\Delta_j)\tilde{\gamma}_g\left(\psi_k(u)\right){\gamma}_g(\Delta_j)-\Delta_j \psi_k(\delta_h(v))\Delta_j\big].
$$
Note that $a = \chi(u)\gamma_g(y_k) - \chi(v)y_k$ and hence
\begin{equation}\label{02-02-21}
\left\|a-a'\right\| \leq \left|\chi(u)\right|\left\|\gamma_g(y_k) - y_k\right\| 
\end{equation}
where
\begin{align*}
a'= \sum_{j=0}^{k-1} \Delta_j(\chi(u) - \chi(v))\Delta_j.
\end{align*}
Set 
$$
b'  =   \sum_{j=k}^{\infty} \Delta_j\left(\tilde{\gamma}_g(\psi_n(u)) - \psi_k(\delta_h (v)\right)\Delta_j.
$$
Since
\begin{align*}
& \left\|{\gamma}_g(\Delta_j)\tilde{\gamma}_g\left(\psi_k(u)\right){\gamma}_g(\Delta_j) - \Delta_j\tilde{\gamma}_g\left(\psi_k(u)\right)\Delta_j\right\| 
\leq 2 \left\|\psi_k(u)\right\| \left\|{\gamma}_g(\Delta_j) - \Delta_j\right\|
\end{align*}
and $\psi_k(u) = \psi_n(u)$, we find that
\begin{equation}\label{02-02-21a}
\left\| b -b'\right\| \leq 2 \left\|\psi_n(u)\right\|\sum_{j=k}^{\infty}  \left\|{\gamma}_g(\Delta_j) - \Delta_j\right\| \leq   2 \left\|\psi_n(u)\right\|\sum_{j=k}^{\infty}  2^{-j},
\end{equation}
thanks to property (vi) in Lemma~\ref{31-01-21a}. Property (iv) in Lemma~\ref{31-01-21a} implies that
$$
b'  =   \sum_{j=k}^{\infty} \Delta_j\left(\tilde{\gamma}_g(\psi_n(u)) - \psi_k(\delta_h(v))\right)(1-y_{j-1})\Delta_j.
$$
Using the above, together with the fact that $m$ commutes with 
$\Delta_j$, we get 
\begin{align*}
&ma'+mb' = \sum_{j=0}^{k-1} \Delta_jm(\chi(u)-\chi(v))\Delta_j \\
& +  \sum_{j=k}^{\infty} \Delta_jm\big(\tilde{\gamma}_g(\psi_n(u)) - \psi_k(\delta_h ( v))\big)(1-y_{j-1})\Delta_j
\end{align*}
when $m \in \{1,d\}$. By property (iii) in Lemma~\ref{31-01-21a} we have that
$$
\left\|m\big(\tilde{\gamma}_g(\psi_n(u)) - \psi_k(\delta_h( v))\big)(1-y_{j-1})\right\| \leq \left\|q_I(m)\big(g \cdot \psi(u) - \psi(h \cdot v)\big)\right\| + \frac{1}{j-1}
$$
for $j \geq k \geq n+1$, since $\tilde{\gamma}_g(\psi_n(u)) - \psi_k(\delta_h( v))$ belongs to $Z_{j-1}$. It follows therefore from Lemma~\ref{01-02-21c} that
$$
\left\|ma'+mb'\right\| \leq \max\{|\chi(u)-\chi(v)|, \left\|q_I(m)(g\cdot \psi(u)-\psi(h\cdot v))\right\|\} + \frac{1}{k-1}
$$
for $m \in \{1,d\}$ when $k \geq n+1$. Combined with \eqref{02-02-21} and \eqref{02-02-21a} we find that
\begin{align*}
& \left\|m\left( g \cdot \psi^k(u)  -  \psi^k(h \cdot v)\right) \right\|\\
& \ \ \ \ \ \ \ \ \ \ \ \ \ = \left\|m (a-a') + m(b -b') + ma'+mb'\right\|\\
& \ \ \ \ \ \ \ \ \ \ \ \ \  \leq \left|\chi(u)\right|\|m\|\left\|\gamma_g(y_k) - y_k\right\| +  2 \|m\|\left\|\psi_n(u)\right\|\sum_{j=k}^{\infty}  2^{-j} \\
& \ \ \  \ \ \ \ \ \ \ \ \ \ \ \ \ \ \  + \max\{|\chi(u)-\chi(v)|, \left\|q_I(m)(g\cdot \psi(u) -\psi(h\cdot v))\right\|\} +  \frac{1}{k-1},
\end{align*}
when $k \geq n+1$. This estimate and property (ii) in Lemma~\ref{31-01-21a} show that \eqref{part7} holds.

(8) Set $u' = u - \chi(u), \ v' = v-\chi(v)$, and 
find $n\in\N$ with $u',v'\in F_n$. Let $\varepsilon > 0$. Since $\psi(u')\psi(v') = q_I(\psi_n(u')\psi_n(v'))$ there is an $x \in I$ such that 
\begin{equation}\label{15-03-21e}
\left\|\psi_n(u')\psi_n(v') + x \right\| \leq \left\|\psi(u')\psi(v')\right\| + \varepsilon.
\end{equation}

Let $t \geq n+2$ and let $k \in \mathbb N$ satisfy $t \in [k,k+1]$. 
Define $s_{k}=t-k$ and $s_j=1$ for $j\geq k+1$. 
Using that $\chi(u') = \chi(v') = 0$ at the first step,
and using that $\Delta_l\Delta_j = 0$ when $|l-j| \geq 2$ at 
the second step, we get
\begin{align*}
\psi^t(u')\psi^t(v')&=\left( \sum_{j = k}^{\infty}\Delta_js_j\psi_n(u') \Delta_j\right) \left(\sum_{j = k}^{\infty}\Delta_js_j\psi_n(v') \Delta_j\right) \\
&= \sum_{l=-1}^1\sum_{j=k+1}^{\infty}  \Delta_j s_j\psi_n(u')\Delta_j \Delta_{j +l} s_{j+l}\psi_n(v')\Delta_{j+l}  \\
& \ \ \ \ \ \ \  +\sum_{l=0}^1  \Delta_{k}s_{k}\psi_n(u')\Delta_{k}\Delta_{k +l} s_{k+l}\psi_n(v')\Delta_{k+l}.
\end{align*}
Using property (v) of Lemma~\ref{31-01-21a}, we get
\begin{equation}\label{15-03-21a}
\left\|\Delta_j \psi_n(v') - \psi_n(v')\Delta_j\right\|
\leq 2^{-j} 
\end{equation}
when $j \geq n$. 
Set \begin{align*}
&a =  \sum_{l=-1}^1\sum_{j=k+1}^{\infty}  \Delta_j s_js_{j+l}\psi_n(u') \Delta_{j +l} \psi_n(v')\Delta_{j+l}\Delta_j \\
& \ \ \ \ \ \ \ \  \ + \sum_{l=0}^1  \Delta_{k}s_{k} s_{k+l}\psi_n(u')\Delta_{k+l}\psi_n(v')\Delta_{k+l}\Delta_{k}.
\end{align*}
Since $k\geq n+1$, we conclude from \eqref{15-03-21a} that 
\begin{equation}\label{02-02-21p}
\left\|\psi^t(u')\psi^t(v') - a\right\| \leq 5 \|\psi_n(u')\|\sum_{j=k+1}^{\infty} 2^{-j+1}.
\end{equation}
Similarly, by setting 
\begin{align*}
%\begin{split}
&b =  \sum_{l=-1}^1\sum_{j=k+1}^{\infty}  \Delta_j s_js_{j+l}\psi_n(u')  \psi_n(v')\Delta_{j+l}^2\Delta_j \\
& \ \ \ \ \ \ \ \  \ + \sum_{l=0}^1  \Delta_{k}s_{k} s_{k+l}\psi_n(u')\psi_n(v')\Delta_{k+l}^2\Delta_{k} \\
& = \sum_{j=k+1}^{\infty}  \Delta_j\left( \sum_{l=-1}^1 s_js_{j+l}\psi_n(u')  \psi_n(v') \Delta_{j+l}^2\right)\Delta_j \\
& \ \ \ \ \ \ \ \  \ +  \Delta_{k}\left(\sum_{l=0}^1 s_{k} s_{k+l} \psi_n(u')\psi_n(v')\Delta_{k+l}^2\right)\Delta_{k},
%\end{split}
\end{align*}
we have 
\begin{equation}\label{02-02-21q}
\|a -b\| \leq 5 \|\psi_n(u')\|\sum_{j=k+1}^{\infty} 2^{-j+1}.
\end{equation}
Set
\begin{equation*}
\begin{split}
&c =\sum_{j=k+1}^{\infty}  \Delta_j\left( \sum_{l=-1}^1 s_js_{j+l}(\psi_n(u')  \psi_n(v') + x) \Delta_{j+l}^2\right)\Delta_j \\
& \ \ \ \ \ \ \ \  \ +  \Delta_{k}\left(\sum_{l=0}^1 s_{k} s_{k+l} (\psi_n(u')\psi_n(v') + x)\Delta_{k+l}^2\right)\Delta_{k}.
\end{split}
\end{equation*} 
Then
\begin{equation}\label{15-03-21c}
\begin{split}
 c- b = \sum_{j=k+1}^{\infty}  \Delta_j\left( \sum_{l=-1}^1 s_js_{j+l}x\Delta_{j+l}^2\right)\Delta_j +  \Delta_{k}\left(\sum_{l=0}^1 s_{k} s_{k+l}x\Delta_{k+l}^2\right)\Delta_{k}.
\end{split}
\end{equation}
Note that
$$
\left\| x\Delta_{j+l}^2 \right\| = \left\|x(y_{j+l+1}-y_{j+l})\right\|,
$$
which will be very small when $j$ is big since $x \in I$. More precisely, there is a $K \in \mathbb N$ such that
$$
\left\| \sum_{l=-1}^1 s_js_{j+l}x\Delta_{j+l}^2\right\| \leq \varepsilon \ \ \text{and} \ \ \left\|\sum_{l=0}^1 s_{k} s_{k+l}x\Delta_{k+l}^2\right\| \leq \varepsilon
$$
when $j \geq k+1$ and $k \geq K$. When we use these estimates in \eqref{15-03-21c} and apply Lemma~\ref{01-02-21c} it follows that 
\begin{equation}\label{15-03-21d}
\| c-b\| \leq \varepsilon,
\end{equation}
provided $k \geq K$.
Note that 
\begin{align*}
 \left\|\sum_{l=-1}^1 s_js_{j+l}(\psi_n(u')  \psi_n(v') +x) \Delta_{j+l}^2\right\| 
&\leq \left\| \psi_n(u')  \psi_n(v') +x \right\| \left\|\sum_{l=-1}^1 s_js_{j+l} \Delta_{j+l}^2\right\| \\
& \leq  \left\| \psi_n(u')  \psi_n(v') +x\right\| \left\|\sum_{l=-1}^1  \Delta_{j+l}^2\right\|  \\
& = \left\| \psi_n(u')  \psi_n(v') +x\right\| \left\|y_{j+2} - y_{j-1}\right\| \\
& \leq  \left\| \psi_n(u')  \psi_n(v') +x\right\|,
\end{align*}
and similarly, 
$$
\left\|\sum_{l=0}^1 s_{k} s_{k+l} (\psi_n(u')\psi_n(v') + x)\Delta_{k+l}^2\right\| \leq \left\| \psi_n(u')  \psi_n(v') + x\right\|.
$$ 
It follows therefore from the definition of $c$ and Lemma~\ref{01-02-21c} that 
$$
\|c\| \leq \left\|\psi_n(u')\psi_n(v') + x\right\|.
$$ 
Combining with \eqref{02-02-21p}, \eqref{02-02-21q}, \eqref{15-03-21d} and \eqref{15-03-21e} we find that
$$
\left\|\psi^t(u')\psi^t(v')\right\| \leq   \left\|\psi(u')\psi(v')\right\| +  10 \|\psi_n(u')\|\sum_{j=k+1}^{\infty} 2^{-j+1} + 2 \varepsilon
 $$
when $t \in [k,k+1]$ for some $k \geq \max\{n+2,K\}$. This proves (8).

(9) Fix $j,n \in \N$ and fix $u \in M_j(F_n)$. 
We make the following abbreviations: 
$\chi' = \chi \otimes \id_{M_j}$, $\psi'=\psi\otimes\id_{M_j}$,
$q_I'=q_I\otimes \id_{M_j}$, $\psi'_k= \psi_k\otimes \id_{M_j}$, $y'_k = y_k \otimes 1_{M_j(\mathbb C)}$
and $\Delta'_k = \Delta_k \otimes 1_{M_j(\mathbb C)}$ for 
$k\in\N$. 
Using property (iv) in Lemma~\ref{31-01-21a}, for 
$k\geq n$ we get
\begin{align*}
(\psi^k \otimes \id_{M_j})(u) & =  \sum_{\ell=0}^{k-1} \Delta'_\ell {\chi'}(u)\Delta'_\ell +  \sum_{\ell=k}^{\infty} \Delta'_\ell \psi'_k(u) \Delta'_\ell \\
&  =  \sum_{\ell=0}^{k-1} \Delta'_\ell {\chi'}(u)\Delta'_\ell +  \sum_{\ell=k}^{\infty} \Delta'_\ell \psi'_n(u) (1-y'_{\ell-1})\Delta'_\ell.
\end{align*}
It follows from Lemma 3.1 in \cite{MT1} that
\begin{equation}\label{02-02-21n}
\left\| (\psi^k \otimes \id_{M_j})(u)\right\| \leq \max \left\{\|\chi'(u)\|, \ \sup_{\ell \geq k} \left\|\psi'_n(u) (1-y'_{\ell-1})\right\|\right\} .
\end{equation}
Since $(y'_k)_{k\in\N}$ is an approximate unit in $M_j(I) = \ker( q_I \otimes \id_{M_j})$, 
$$
\lim_{\ell \to \infty} \left\|\psi'_n(u) (1-y'_{\ell-1})\right\|  = \left\|q'_I(\psi'_n(u))\right\| = \left\|\psi'(u)\right\|.
$$
It follows therefore from \eqref{02-02-21n} that
$$
\limsup_{k \to \infty}\left\| (\psi^k \otimes \id_{M_j})(u)\right\| \leq \max\big\{\left\|\chi'(u)\right\|, \left\|\psi'(u)\right\|\big\}.
$$
Thanks to \eqref{02-02-21k} this proves (9).

Finally, for the last statement, assume that a unital completely positive linear lift $\widetilde{\psi}$ of $\psi$ exists, and that we have chosen all $\psi_n$'s to be equal to $\widetilde{\psi}$. It then follows immediately 
from their
definition that 
each $\psi^t$ will be a linear, completely positive contraction.
This completes the proof. \end{proof}

Next, we introduce some terminology relative to families $\Theta=(\Theta_t)_{t\in [1,\infty)}$ of maps
$\Theta_t\colon A\to B$ between $C^*$-algebras.
The cone of positive elements in a $C^*$-algebra $A$ will be denoted by $A_+$. Given a self-adjoint element $a =a^*\in A$, there are unique elements  $a^+, a^-\in A_+$ such that $a^+a^- = 0$ and $a = a^+  - a^-$. This notation is used in (e) of the following

\begin{defn}\label{15-01-21} Let $A$ and $B$ be $C^*$-algebras. 
We say that a collection 
$\Theta=(\Theta_t)_{t\in [1,\infty)}$ of maps
$\Theta_t\colon A\to B$ is a
\emph{continuous path}, when $(\Theta_t)_{t\in [1,\infty)}$ is an 
equicontinuous 
family of maps and for every $a\in A$, the assignment 
$[1,\infty)\to A$ given by 
$t\mapsto\Theta_t(a)$ is continuous. 
Additionally, we say that $\Theta$ is
\begin{itemize}
\item[(a)] \emph{unital}, if $\Theta_t(1) = 1$ for all $t$;
\item[(b)] \emph{self-adjoint}, if $\Theta_t$ is self-adjoint for all  $t \in [1,\infty)$;
\item[(c)] \emph{asymptotically linear}, if
$$\lim_{t \to \infty} \left\| \Theta_t(\lambda_1 a_1 + \lambda_2 a_2)-\lambda_1 \Theta_t(a_1) - \lambda_2 \Theta_t(a_2)\right\|=0
$$ 
for all $\lambda_1,\lambda_2 \in \mathbb C$ and all $a_1,a_2 \in A$;
\item[(d)] \emph{asymptotically contractive}, if $\limsup\limits_{t \to \infty} \left\| \Theta_t(a)\right\|\leq\|a\|$ for $a \in A$; 
%\item[e)] $\Theta$ is an asymptotic isometry when $\lim_{t \to \infty} \left\| \Theta_t(a)\right\|= \|a\|$ for all $a \in A$, 
\item[(e)] \emph{asymptotically positive}, if $\Theta$ is self-adjoint and
$$
\lim_{t \to \infty} \Theta_t(a)^-=0
$$
for all positive elements $a \in A_+$;
\item[(f)] \emph{asymptotically completely positive}, if the continuous family
$$
\Theta \otimes \id_{M_n} = \left(\Theta_t \otimes \id_{M_n}\right)_{t \in [1,\infty)}\colon M_n(A)\to  M_n(B)
$$ 
is asymptotically linear and asymptotically positive for all $n$.
\end{itemize} 

\end{defn}

The following is an analog, in the context of continuous families of maps between $C^*$-algebras, of the well-known fact that
self-adjoint, linear maps are positive.

\begin{lemma}\label{16-01-21gx} Let $A$ and $B$ be unital $C^*$-algebras, and let $\Theta \colon A \to  B$ be a continuous family of maps which is self-adjoint, unital, asymptotically linear and an asymptotic contraction. Then $\Theta$ is asymptotically positive.
\end{lemma} 
\begin{proof}
Let $a \in A_+$. In order to reach a contradiction, assume that $\Theta_t(a)^-$ does not converge to $0$ as $t \to \infty$. Then there are $\varepsilon >0$ and a sequence $(t_n)_{n\in\N}$ in $[1,\infty)$ such that $\lim_{n \to \infty} t_n = \infty$ and $\left\|\Theta_{t_n}(a)^-\right\|\geq\varepsilon$ for all $n\in\N$. Find states $\omega_n$ on $B$, for $n\in\N$, such that $\omega_n\left(\Theta_{t_n}(a)^+\right) =0$ and $\omega_n\left(\Theta_{t_n}(a)^-\right) \geq \varepsilon$ for all $n\in\N$. Using that $\Theta_{t}$ is asymptotically linear, we have. 
$$
\lim_{n \to \infty} \left\| \Theta_{t_n}(\|a\| - a)-(\|a\| - \Theta_{t_n}(a))\right\|=0.
$$ 
Similarly, using that $\Theta_t$ is an asymptotic contraction
at the second step, we have
$$
\limsup_{t \to \infty} \left\|\Theta_t( \| a\| - a)\right\|\leq\|\|a\| - a\|\leq\|a\|.
$$
We deduce that there is $n\in\N$ such that
\begin{align*}
& \|a\|+\frac{\varepsilon}{3}\geq 
\omega_n (\Theta_{t_n}(\|a\|-a))\geq\omega_n (\|a\| - \Theta_{t_n}(a))-\frac{\varepsilon}{3} \\
& =  \|a\|  + \omega_n(\Theta_{t_n}(a)^-)- \frac{\varepsilon}{3} \ \geq \|a\| + \frac{2\varepsilon}{3},
\end{align*}
which is a contradiction, and thus $\Theta$ is asymptotically positive.
\end{proof}

We now make explicit what we meant at the beginning of this 
subsection when we said that we want out lifts for $\psi$
to have properties that are
``as close as possible to those of $\psi$''. 
Among others, we would like $\psi'$ to be as close as possible to being completely positive, and if $\psi$ is equivariant, then we would like something similar to be true to for $\psi'$. More generally, we would like the failure of equivariance for $\psi'$ to be approximately controlled by the failure of equivariance for $\psi$. These notions are made precise in the following definition, in a slightly more general setting.

\begin{defn}\label{25-02-21a} Let $(A,\alpha)$, $(B,\beta)$ and $(S,\delta)$ be $G$-$C^*$-algebras, and let $q\colon (A,\alpha) \to (B,\beta)$ an equivariant $*$-homomorphism. Let $\psi\colon S\to B$ be a linear completely positive contraction, and let $\Theta=(\Theta_t)_{t\in [1,\infty)} \colon S\to  A$ be a 
continuous path of maps. 
\begin{enumerate}
\item  We say that $\Theta$ is an \emph{asymptotically $(G,\psi)$-equivariant} lift of $\psi$ when 
\begin{itemize}
\item $q\circ \Theta_t = \psi$ for all $t \in [1,\infty)$;
\item $\Theta$ is asymptotically completely positive;
\item for all
$s\in S$ we have 
$$
\lim\limits_{t \to \infty} \| g \cdot \Theta_t(s)-\Theta_t(h \cdot s)\| = \|g \cdot \psi(s) - \psi(h\cdot s)\|,
$$ 
uniformly for $g$ and $h$ in compact subsets of $G$; and
\item for all $s \in S$, $g\in G$ and $\varepsilon>0$, 
there is an open neighborhood $W$ of $g$ such that for all $h\in W$ we have
$$ 
\sup_{t \in [1,\infty)} \left\|h \cdot \Theta_t(s)-g\cdot \Theta_t(s)\right\|\leq\varepsilon.
$$
\end{itemize}
\item We say that $\Theta$ is a \emph{completely positive asymptotically $(G,\psi)$-equivariant} lift of $\psi$ when
\begin{itemize}
\item $q\circ \Theta_t = \psi$ for all $t \in [1,\infty)$;
\item $\Theta_t$ is a linear completely positive contraction for all $t \in [1,\infty)$; and
\item for all
$s\in S$ we have 
$$
\lim\limits_{t \to \infty} \| g \cdot \Theta_t(s)-\Theta_t(h \cdot s)\| = \|g \cdot \psi(s) - \psi(h\cdot s)\|,
$$ 
uniformly for $g$ and $h$ in compact subsets of $G$.
\end{itemize}
\end{enumerate}
\end{defn}

We are now ready to prove the main result of this section.
The main technical difficulty in proving it is making sure that 
the family $(\theta_t(u))_{t\in [1,\infty)}$ is uniformly 
bounded for all $u\in U$, not just for all $u$ in a dense subalgebra. (If we were only interested in the latter, then we
could take $\theta_t$ to be the map $\psi^t$ constructed in 
Lemma~\ref{01-02-21d}.)

\begin{thm}\label{02-02-21r} 
Let $G$ be a locally compact, second countable 
group, let $(I,\gamma)$ be a $\sigma$-unital $G$-algebra,
and let $d\in I$ be a strictly positive element. Denote by 
$q_I\colon M(I)\to Q(I)$ the canonical quotient map.
Let $(U,\delta)$ be a separable, unital $G$-algebra, let $\chi\colon U\to \mathbb{C}$ be a $G$-invariant state, and let $\psi\colon U\to Q(I)_{\overline{\gamma}}$ be a unital, completely positive linear map.

Then there is a unital asymptotically $(G,\psi)$-equivariant lift 
\[\Theta=(\Theta_t)_{t\in [1,\infty)}\colon U \to  M(I)_{\tilde{\gamma}}\] of $\psi$ with the following additional properties:
 \begin{itemize}
 \item[(i)] For all $u,v \in U$, every compact subset $K \subseteq G$ and every $\varepsilon > 0$, there is $T \geq 1$ such that
 \begin{align*} & 
\sup_{t\geq T} \left\|g \cdot \Theta_t(u) - \Theta_t(h\cdot v) \right\|\\
 & \ \ \ \ \ \ \ \  \leq
\max \left\{|\chi(u-v)|, \|g \cdot \psi(u) - \psi(h\cdot v)\|\right\} + \varepsilon \ 
\end{align*}
for all $g,h \in K$,
\item[(ii)] for all $u \in U$ and all $x \in I$, we have 
$\lim\limits_{t \to \infty} \Theta_t(u)x   =  \chi(u)x$,
\item[(iii)] for all $n\in\N$ and all $u \in M_n(U)$, we have $$\lim_{t\to \infty} \left\|(\Theta_t \otimes \id_{M_n})(u)\right\| = \max\big\{\left\|(\psi \otimes \id_{M_n})(u)\right\|, \left\|(\chi \otimes \id_{M_n})(u)\right\|\big\},$$ and
\item[(iv)] for $u,v\in \ker(\chi)$, we have 
$\lim_{t \to \infty} \|\Theta_t(u)\Theta_t(v) \| = \|\psi(u)\psi(v)\|.$
\end{itemize}

Finally, if a unital linear completely positive lift for $\psi$ exists, then the maps $\Theta_t$ above can be chosen to additionally be unital and completely positive.
\end{thm}
\begin{proof} Consider the $C^*$-algebra
$$
\mathcal{A}  =  \left\{f \in C_b\left([1,\infty), M(I)_{\tilde{\gamma}}\right)\colon f(1) -f(t) \in I \mbox{ for all } t\in [1,\infty) \right\},
$$
and define an action $\mu\colon G\to\Aut(\mathcal{A})$ by
$\mu_g(f)(t)=\tilde{\gamma}_g(f(t))$ for all $g\in G$, all 
$f\in\mathcal{A}$ and all $t\in [1,\infty)$.
Set
$$
\mathcal{A}_0 = \mathcal{A} \cap C_0\left([1,\infty),M(I)_{\tilde{\gamma}}\right),
$$
which is an ideal in $\mathcal{A}$. Note that $\mathcal{A}_0 = C_0([1,\infty),I)$, that $\mu_g(\mathcal A_0) = \mathcal A_0$ for all $g \in G$, and that $\mathcal A_0$ is contained in the continuous part $\mathcal A_\mu$ of $\mathcal A$. 
We denote by $\pi\colon \mathcal{A} \to \mathcal{A}/\mathcal{A}_0$ 
the quotient map and let $\overline{\mu}\colon G\to \Aut(\mathcal{A}/\mathcal{A}_0)$ be the action defined such that $\overline{\mu}_g \circ \pi = \pi \circ \mu_g$
for all $g\in G$. 

Let $\psi^t\colon U\to M(I)_{\tilde{\gamma}}$ be a family of maps satisfying the conclusion of Lemma~\ref{01-02-21d}.
By (7) in said lemma (with $g=e$ and $v=0$), we get $\sup_{t \in [1,\infty)} \left\|\psi^t(u)\right\| < \infty$ when $u \in  U^{(0)}$. In combination with (2) and (5) of 
Lemma~\ref{01-02-21d}, this gives us a map
 $\Phi_0\colon U^{(0)} \to \mathcal{A}$ defined by 
 $$
\Phi_0(u)(t) = \psi^t(u)
$$
for all $u\in U^{(0)}$ and all $t\in [1,\infty)$. Set $\Psi_0=\pi\circ\Phi_0\colon U^{(0)}\to \mathcal{A}/\mathcal{A}_0$.
It follows then from (7) in
Lemma~\ref{01-02-21d} that
$$
\left\|\Psi_0(u) - \Psi_0(v)\right\| \leq \max\{|\chi(u-v)|, \left\|\psi(u-v)\right\|\} 
$$
for all $u,v \in U^{(0)}$. Since $U^{(0)}$ is dense in $U$, it follows from this estimate that $\Psi_0$ extends by continuity to a continuous map $\Psi \colon U \to \mathcal{A}/\mathcal{A}_0$ with the property that
$$
\left\|\Psi(u) - \Psi(v)\right\|  \leq \max\{|\chi(u-v)|, \left\|\psi(u-v)\right\|\} 
$$
for all $u,v \in U$. Note that $\Psi$ is self-adjoint because each $\psi^t$ is and $(U^{(0)})^* = U^{(0)}$. 

Let $u_1,\ldots,u_n \in U^{(0)}$ and 
let $\lambda_1,\ldots,\lambda_n\in \mathbb{Q}[i]$, so
that $ \sum_{j=1}^n \lambda_j u_j$ belongs to $U^{(0)}$ by  \eqref{04-01-21g}. Using (4) of 
Lemma~\ref{01-02-21d} and linearity of $\pi$, we have
$$
\Psi\Big(\sum_{j=1}^n \lambda_j u_j\Big) =\Psi_0\Big(\sum_{j=1}^n \lambda_j u_j\Big) = \sum_{j=1}^n \lambda_j \Psi_0(u_j)= \sum_{j=1}^n \lambda_j \Psi(u_j).
$$
Hence $\Psi$ is $\mathbb{Q}[i]$-linear on $U^{(0)}$, 
and by continuity it is linear on all of $U$. 
Note that $\Psi$ is unital by (1) of 
Lemma~\ref{01-02-21d}.

Let $k \in \mathbb N$. Then $\Psi\otimes \id_{M_k}$ is self-adjoint because $\Psi$ is and it follows from (9) of 
Lemma~\ref{01-02-21d} that $\Psi \otimes \id_{M_k}$ is a contraction, since $\chi \otimes \id_{M_k}$ and $\psi\otimes \id_{M_k}$ are. Since $\Psi\otimes \id_{M_k}$ is also unital, we deduce from Lemma~\ref{16-01-21gx} that $\Psi\otimes \id_{M_k}$ is positive. We conclude that $\Psi$ is completely positive.

Let $g,h \in G$ and $u,v \in U^{(0)}$. Since $G^{(0)}$ is
dense in $G$, we can find a sequence $(h_n)_{n\in\N}$ in $G^{(0)}$ such that $\lim_{n \to \infty} h_n = h$. By \eqref{04-01-21g}, we 
have $h_n\cdot v\in U^{(0)}$ for all $n\in\N$.
Taking $h_n\cdot v\in U^{(0)}$ in place of $v$ and $e$ in 
place of $h$ in (7) of 
Lemma~\ref{01-02-21d}, we get
\begin{align*}
\left\|g\cdot\Psi(u) - \Psi(h_n\cdot v)\right\| &= \limsup_{t \to \infty} \left\|g\cdot \psi^t(u) - \psi^t(h_n \cdot v)\right\| \\
& \leq  \max \left\{|\chi(u-v)|, \|g \cdot \psi(u) - \psi(h_n\cdot v)\|\right\}.
\end{align*}
By continuity of $\Psi$ and $\psi$, we can take the limit on $n$ to get 
\begin{equation}\label{25-02-21} 
\|g\cdot \Psi(u) - \Psi(h\cdot v)\|  \leq  \max \left\{|\chi(u-v)|, \|g \cdot \psi(u) - \psi(h\cdot v)\|\right\}.
\end{equation}
Since $U^{(0)}$ is dense in $U$ and $\Psi$ and $\psi$ are continuous, we conclude that \eqref{25-02-21} above holds for all $u, v \in U$. It follows, in particular, that 
$$
\lim_{g \to e} \overline{\mu}_g (\Psi(u)) = \Psi(u),
$$
showing that $\Psi$ takes values in the continuous part $(\mathcal{A}/\mathcal{A}_0)_{\overline{\mu}}$ of $\mathcal{A}/\mathcal{A}_0$. It follows from \cite{Br} that $\pi(\mathcal{A}_\mu) = (\mathcal{A}/\mathcal{A}_0)_{\overline{\mu}}$ and then from the Bartle-Graves selection theorem, \cite{BG}, that there is a continuous map $\Psi'\colon U \to \mathcal{A}_\mu$ such that $\pi \circ \Psi' = \Psi$. By substituting $\Psi'$ with $\Psi' - \Psi'(0)$ we can assume that $\Psi'(0) = 0$, and by substituting $\Psi'(u)$ with $\Psi'(u) + \chi(u)(1-\Psi'(1))$, also that $\Psi'(1) = 1$. By substituting $\Psi'(u)$ with $\frac{1}{2}\left(\Psi(u) + \Psi(u^*)^*\right)$, we may assume that $\Psi'$ is self-adjoint. For $t\in [1,\infty)$, let $\ev_t\colon \mathcal{A} \to M(I)_{\tilde{\gamma}}$ be evaluation at $t$ and set
$$
\Theta_t = \ev_t \circ \Psi'\colon U\to M(I)_{\tilde{\gamma}}.
$$ 
Then $\Theta=(\Theta_t)_{t\in [1,\infty)}$ is a unital and self-adjoint continuous family of maps $U \to  M(I)_{\tilde{\gamma}}$. 
In the remainder of the proof, we check that $\Theta$ has the properties required in the statement of the lemma. 

Fix $t\in [1,\infty)$. 
To show that $q_I \circ \Theta_t = \psi$, it suffices to check the identity on $U^{(0)}$. Fix $u\in U^{(0)}$. 
Then $\Psi'(u) - \Phi_0(u) \in \mathcal{A}_0$ and hence $\Theta_t(u) - (\ev_t \circ \Phi_0)(u)$ belongs to $I$. 
Since $\ev_t(\Phi_0(u)) = \psi^t(u)$, it follows 
that $q_I(\Theta_t(u)) = q_I(\psi^t(u))$, which equals $\psi(u)$
by (5) of 
Lemma~\ref{01-02-21d}. 

To check that $\Theta$ is asymptotically linear, let $\lambda_1,\lambda_2 \in \mathbb C$ and $u_1,u_2\in U$. Using the linearity of $\Psi$ we find that
\begin{align*}
&\limsup_{t\to \infty}  \left\| \Theta_t  (\lambda_1 u_1 + \lambda_2 u_2)- \lambda_1 \Theta_t(u_1) - \lambda_2 \Theta_t (u_2)\right\|\\
& \ \ \ = \limsup_{t \to \infty}\left\|\ev_t\big(\Psi'(\lambda_1 u_1+\lambda_2u_2)- \lambda_1 \Psi'(u_1)- \lambda_2 \Psi'(u_2)\big) \right\| \\
& \ \ \ = \left\|\pi\big(\Psi'(\lambda_1 u_1+\lambda_2u_2)- \lambda_1 \Psi'(u_1)- \lambda_2 \Psi'(u_2)\big) \right\| \\
& \ \ \ = \left\|\Psi(\lambda_1 u_1+\lambda_2u_2)- \lambda_1 \Psi(u_1)- \lambda_2 \Psi(u_2) \right\|= 0,
\end{align*}
as desired. To show that $\Theta$ is asymptotically completely positive, we fix $k \in \mathbb N$ and will prove that $\Theta \otimes \id_{M_k}$ is asymptotically positive. Note that $\Theta \otimes \id_{M_k}$ is self-adjoint, unital and asymptotically linear. By Lemma~\ref{16-01-21gx}, it suffices to show that
\begin{equation}\label{18-01-21bx}
\limsup_{t \to \infty} \|(\Theta_t \otimes \id_{M_k})(u)\| \leq \max \{\|(\psi\otimes \id_{M_k})(u)\|,  \|(\chi\otimes \id_{M_k})(u)\|\}
\end{equation} 
for all $u \in M_k(U)$, since $\psi\otimes \id_{M_k}$ and $\chi \otimes \id_{M_k}$ are contractions. This estimate follows from equicontinuity of the continuous family $\Theta \otimes \id_{M_k}$ together with (8) of 
Lemma~\ref{01-02-21d}, since $M_k(U^{(0)})$ 
is dense in $M_k(U)$. We conclude that $\Theta$ is is asymptotically completely positive.

To obtain (i), let $u,v \in U$ and let $K$ be a compact subset of $G$. Let $\varepsilon > 0$. Since $\Psi'(u) \in \mathcal{A}_\mu$ and since $\Theta$ is equicontinuous, for every 
$g,h\in K$ 
there are open neighborhoods $V_g$ of $g$ and $W_h$ of $h$ such that 
$$
\left\| g'\cdot \Theta_t(u) - \Theta_t(h'\cdot v)\right\| \leq \left\| g\cdot \Theta_t(u) - \Theta_t(h\cdot v)\right\| + \frac{\varepsilon}{3}
$$
for all $(g',h') \in V_g \times W_h$ and for all $t \in [1,\infty)$. By shrinking $V_g$ and $W_h$ we can also arrange that
$$
\left\| g\cdot \psi(u) - \psi( h \cdot v)\right\| \leq \left\| g'\cdot \psi(u) - \psi( h' \cdot v)\right\| + \frac{\varepsilon}{3}
$$
for all $(g',h') \in V_g \times W_h$. By compactness of $K  \times K$, there is a finite set $ K_0 \subseteq K \times K$ such that $\{V_g \times W_h\colon (g,h)\in K_0\}$ covers 
$K\times K$. Since $K_0$ is finite, it follows from \eqref{25-02-21} that there is a $T \geq 1$ such that
$$
\left\| g\cdot \Theta_t(u) - \Theta_t(h\cdot v)\right\| \leq \max\{|\chi(u-v)|, \left\| g\cdot \psi(u) - \psi( h \cdot v)\right\|\} + \frac{\varepsilon}{3}
$$
for all $(g,h) \in K_0$ when $t \geq T$. 
An easy application of the triangle inequality then yields
$$
\left\| g\cdot \Theta_t(u) - \Theta_t(h\cdot v)\right\| \leq \max\{|\chi(u-v)|, \left\| g\cdot \psi(u) - \psi( h \cdot v)\right\|\} + \varepsilon
$$
for all $g,h\in K$ when $t \geq T$, which establishes (i). On the other hand it follows from \eqref{25-02-21} that
\begin{align*}
& \limsup_{t \to \infty} \left\|g \cdot \Theta_t(u) - \Theta_t(h \cdot u)\right\| = \left\|g \cdot \Psi(u) -\Psi(h \cdot u)\right\|\\
& \leq \left\|g \cdot \psi(u) -\psi(h \cdot u)\right\|,
\end{align*}
for all $g,h \in G, \ u \in U$, which combined with (i) shows that $\Theta$ satisfies the condition in the third bullet of (1) in Definition \ref{25-02-21a}. The fourth bullet of (1) in Definition \ref{25-02-21a} follows because $\Psi(U) \subseteq \mathcal A_{\mu}$.

It remains now only to establish the last three items, (ii)-(iv), in the statement. For (ii) consider first an element $u\in U^{(0)}$. Taking 
$m=d$, $g = h = e$ and $v = \chi(u)\in U^{(0)}$ in (7) of Lemma~\ref{01-02-21d} gives 
$$
\lim_{t \to \infty} \left\|d(\psi^t(u) - \chi(u))\right\| = 0.
$$ 
Hence, upon identifying $d$ with the element of $\mathcal{A}$ which is constant with value $d$, we have $\pi(d{\Psi'}(u) ) = \chi(u)\pi(d)$. This identity means that $d\Psi'(u) - \chi(u)d$ belongs to $\mathcal{A}_0$. It follows that 
$$\lim_{t \to \infty} \Theta_t(u)d - \chi(u)d = \lim_{t \to \infty} \left(d\Theta_t(u^*) - \chi(u^*)d\right)^* = 0\ .
$$  
Since $\overline{dI} = I$ and $\left\|\Theta_t(u)\right\|$ is uniformly bounded, it follows that 
\begin{equation}\label{15-03-21}
\lim_{t \to \infty} \Theta_t(u)x= \chi(u)x
\end{equation} 
for all $x \in I$. Since $\Theta$ is an equicontinuous family of maps this conclusion extends to all elements $u \in U$.

To obtain (iii) from (ii), fix $u\in M_n(U)$ and $\varepsilon>0$. 
Since $\Theta_t \otimes \id_{M_n}$ is a lift of $\psi \otimes \id_{M_n}$, we have
\begin{align}\label{24-01-21ax}
\|(\psi \otimes \id_{M_n})(u)\| \leq \|(\Theta_t \otimes \id_{M_n})(u)\| 
\end{align}
for all $t\in [1,\infty)$. Since $M_n(M(I))\cong M(M_n(I))$, we can
choose $x \in M_n(I)$ with $\|x\| \leq 1$ such that
$$
\left\|(\chi \otimes \id_{M_n})(u)x\right\| > \left\|(\chi \otimes \id_{M_n})(u)\right\| -  \varepsilon.
$$
It follows from \eqref{15-03-21} that 
$$
\lim_{t \to \infty} (\Theta_t\otimes \id_{M_n}(u))x = (\chi\otimes \id_{M_n})(u)x,
$$ 
and hence 
$$
\left\|(\Theta_t\otimes \id_{M_n})(u)\right\| \geq  \left\|(\Theta_t\otimes \id_{M_{n}})(u)x\right\|  >  \left\|(\chi \otimes \id_{M_n})(u)\right\| -  \varepsilon 
$$
for all $t$ big enough. In combination with 
\eqref{24-01-21ax}, this shows that
$$
\liminf_{t \to \infty} \left\|(\Theta_t\otimes \id_{M_n})(u)\right\|  \geq  \max\big\{\left\|(\psi \otimes \id_{M_n})(u)\right\|,   \left\|(\chi \otimes \id_{M_n})(u)\right\| \big\}.
$$
In combination with \eqref{18-01-21bx} this gives us (iii).

Finally, from (8) of Lemma~\ref{01-02-21d} we get the estimate
$$
\limsup_{t \to \infty} \big\|\Theta_t(u - \chi(u))\Theta_t(v - \chi(v)) \big\| \leq \left\|\psi(u-\chi(u))\psi(v- \chi(v))\right\|
 \ 
$$
for all $u,v \in U^{(0)}$. This estimate extends to all $u,v \in U$ by equicontinuity of $\Theta$, implying that
$$
\limsup_{t \to \infty} \left\|\Theta_t(u)\Theta_t(v)\right\| \leq \left\|\psi(u)\psi(v)\right\|
$$
when $u,v \in \ker \chi$. Since $\left\|\psi(u)\psi(v)\right\| = \left\|q_I(\Theta_t(u)\Theta_t(v))\right\| \leq \left\|\Theta_t(u)\Theta_t(v)\right\|$ for all $t$, we have established (iv).

The last assertion in the theorem follows from the last assertion
of Lemma~\ref{01-02-21d}, once we choose the maps $\psi^t$ to also be
linear, unital and completely positive.
\end{proof}

\section{Asymptotic lifts}\label{Alifts}

In this section, we prove our main result, Theorem~\ref{04-01-21d}, which implies Theorem~A from the introduction. Before doing this,
we first prove an auxiliary result in the unital setting, Proposition~\ref{18-01-21h}, assuming
that $S, A$ and $\psi$ are unital, and that there exists a 
$G$-invariant state on $S$. 
This 
result will follow from the work done in the previous section by using the Busby invariant to relate an arbitrary extension to a Calkin extension. The general case will be reduced to this 
one by taking (forced) unitizations.

\begin{prop}\label{18-01-21h}
Let $G$ be a second countable, locally compact group, let 
\begin{equation}\label{G-extensionxxx}
\begin{xymatrix}{
0 \ar[r] & (I, \gamma) \ar[r]^\iota & (A,\alpha) \ar[r]^-q  & (B,\beta)  \ar[r] & 0 }
\end{xymatrix}
\end{equation}
be an extension of $G$-algebras.
Let $(S,\delta)$ be a 
separable $G$-algebra and let $\psi\colon S \to B$ be a linear completely positive contraction.
Assume moreover that $A$, $S$ and $\psi$ are unital, and that
there is a $G$-invariant state $\chi$ on $S$.

Then there is a unital asymptotically $(G,\psi)$-equivariant lift $\Theta=(\Theta_t)_{t\in [1,\infty)} \colon S\to  A$ of $\psi$ with the following additional properties:

\begin{itemize}
\item[(a)] For all $s,s' \in S$, every compact subset $K \subseteq G$ and every $\varepsilon > 0$, there is $T \geq 1$ such that
 \begin{align*} & 
\sup_{t\geq T} \left\|g \cdot \Theta_t(s) - \Theta_t(h\cdot s') \right\|\\
 & \ \ \ \ \ \ \ \  \leq
\max \left\{|\chi(s-s')|, \|g \cdot \psi(s) - \psi(h\cdot s')\|\right\} + \varepsilon \ 
\end{align*}
for all $g,h \in K$,
\item[(b)] for all $n\in\N$ and all $s \in M_n(S)$, we have 
$$\lim_{t\to \infty} \left\|(\Theta_t \otimes \id_{M_n})(s)\right\| = \max\big\{\left\|(\psi \otimes \id_{M_n})(s)\right\|, \left\|(\chi \otimes \id_{M_n})(s)\right\|\big\},$$ 
\item[(c)] for $s,s'\in \ker(\chi)$ we have 
$\lim_{t \to \infty} \|\Theta_t(s)\Theta_t(s') \| = \|\psi(s)\psi(s')\|,$ and
\item[(d)] if $I$ is $\sigma$-unital, we can also arrange that 
$\lim\limits_{t \to \infty} \Theta_t(s)x   =  \chi(s)x$
for all $s \in S$ and all $x \in I$.
\end{itemize}

Moreover, if a linear, unital completely positive lift for $\psi$ exists, then the continuous family $\Theta$ as above can be chosen so that, in addition, each map $\Theta_t$ is unital and completely positive.
\end{prop}
\begin{proof}
We first explain how to reduce to the case where $I$ is $\sigma$-unital (in fact, we reduce to the case where $A$ is separable). Since $S$ is separable and $G$ is second countable, we can choose separable $G$-subalgebras $B_0 \subseteq B$ and $A_0\subseteq A$ with $q(A_0)=B_0$ such that $\psi(S) \subseteq B_0$. When $I$ is $\sigma$-unital, we may arrange that $A_0$ contains a strictly positive element of $I$. In either case, 
$I_0 = I \cap A_0$ is separable, $(I_0,\gamma)$ is a $G$-algebra, and the following diagram commutes:
\begin{equation*}\label{E2}
\begin{xymatrix}{
0 \ar[r] & (I_0,\gamma) \ar@{^{(}->}[d]  \ar[r]  & (A_0,\alpha) \ar@{^{(}->}[d] \ar[r]^-{q} &  (B_0,\beta) \ar@{^{(}->}[d]  \ar[r] & 0 \\
0 \ar[r] & (I,\gamma) \ar[r]  & (A,\alpha) \ar[r]^-{q}  &  (B,\beta) \ar[r] & 0. }
\end{xymatrix}
\end{equation*}
Since $\psi$ takes values in $B_0$ it follows from the preceding discussion that there is an asymptotically $(G,\psi)$-equivariant lift of $\psi$ to $A_0$, and by the above diagram this is also a lift of $\psi$ to $A$. Thus, it suffices to prove the result
for the extension $0\to I_0\to A_0\to B_0\to 0$; in other words,
we may assume from now on that $A$ is separable.

Define $r\colon A \to M(I)$ by $\iota(r(a)x) = a\iota(x)$ for all $a\in A$ and all $x \in I$. Then $r\colon (A,\alpha) \to (M(I)_{\tilde{\gamma}},\tilde{\gamma})$ is a unital $G$-equivariant $*$-homomorphism, and thus $q_I \circ r\colon  (A,\alpha) \to (Q(I)_{\overline{\gamma}},\overline{\gamma})$ is $G$-equivariant as well. Since $\iota(I) \subseteq \ker(q_I \circ r)$, we obtain an equivariant $*$-homomorphism $\phi\colon (B,\beta) \to (Q(I)_{\overline{\gamma}},\overline{\gamma})$ making the following diagram commute:
\begin{equation*}
\centerline{\begin{xymatrix}{
A \ar[d]^-r  \ar[r]^-{q}  & B \ar[d]^-\phi\\
M(I)_{\tilde{\gamma}} \ar[r]_{q_I}  &  Q(I)_{\overline{\gamma}}. }
\end{xymatrix}}
\end{equation*}
The map $\phi$ is known as the Busby invariant of the extension \eqref{G-extensionxxx}. Set
$$
E= \left\{ (m,b) \in M(I)_{\tilde{\gamma}} \oplus B\colon \ q_I(m)= \phi(b) \right\}.
$$
Define $\mu\colon G \to \Aut(E)$ by $\mu_g(m,b) = \left(\tilde{\gamma}_g(m), \beta_g(b)\right)$ for all $g\in G$
and all $(m,b)\in E$. Then $(E,\mu)$ is a $G$-algebra. 
Define $\ast$-homomorphisms $\iota'\colon I \to E$, $p\colon E\to B$,
and $\xi\colon A\to E$
by $\iota'(x) = (x,0)$, $p(m,b) = b$, and $\xi(a) = (r(a), q(a))$, for $x\in I$, $(m,b)\in E$ and $a\in A$. Then
\begin{equation*}\label{E3}
\begin{xymatrix}{
0 \ar[r] & (I,\gamma) \ar@{=}[d] \ar[r]^-\iota  & (A,\alpha) \ar[d]^-{\xi}\ar[r]^-{q} &  (B,{\beta})\ar@{=}[d] \ar[r] & 0 \\
0 \ar[r] & (I,\gamma) \ar[r]_-{\iota'}  & (E,\mu) \ar[r]_-{p}  &  (B,{\beta}) \ar[r] & 0 }
\end{xymatrix}
\end{equation*}
is a commuting diagram of $G$-algebras with exact rows. In particular, $\xi$ is an isomorphism of $G$-algebras. 
Let $\Theta'=(\Theta'_t)_{t\in [1,\infty)} \colon S\to  M(I)_{\tilde{\gamma}}$ be a unital asymptotically $(G,\phi \circ \psi)$-equivariant lift of the unital completely positive contraction $\phi \circ \psi \colon S \to Q(I)_{\overline{\gamma}}$
satisfying the conclusion of Theorem~\ref{02-02-21r}. 
Define $\Theta=(\Theta_t)_{t\in [1,\infty)} \colon S\to  A$ by
$$
\Theta_t(s)= \xi^{-1}\left((\Theta'_t(s), \psi(s))\right) \ 
$$
for all $t\in [1,\infty)$ and all $s\in S$.
One readily checks that $\Theta$ is a unital asymptotically $(G,\psi)$-equivariant lift of $\psi$ to $A$ satisfying the conclusion of this theorem.

Finally, if a unital linear completely positive lift for $\psi$ 
exists, then a unital linear completely positive lift for 
$\phi\circ\psi$ exists as well. Applying the last part of the 
statement of Theorem~\ref{02-02-21r}, we may choose $\Theta'$
so that $\Theta'_t$ is additionally unital and completely positive,
in which case the same is true for $\Theta$. 
\end{proof}

We now present the main result of this work.

\begin{thm}\label{04-01-21d}
Let $G$ be a second countable, locally compact group, let 
\[
\begin{xymatrix}{
0 \ar[r] & (I, \gamma) \ar[r]^\iota & (A,\alpha) \ar[r]^-q  & (B,\beta)  \ar[r] & 0 }
\end{xymatrix}
\]
be an extension of $G$-algebras.
Let $(S,\delta)$ be a 
separable $G$-algebra and let $\psi\colon S \to B$ be a linear completely positive contraction.
Then there is a lift $\Theta=(\Theta_t)_{t\in [1,\infty)} \colon S\to  A$ of $\psi$ with the following properties:
\begin{itemize}
 \item[(a)] For all $s,s'\in S$, we have 
 $$\lim_{t \to \infty} \left\| g \cdot \Theta_t(s) - \Theta_t(h\cdot s')\right\| = \left\|g \cdot \psi(s) - \psi(h\cdot s')\right\|,$$
 uniformly for $g,h$ in compact subsets of $G$;
 \item[(b)] for all $n\in\N$ and all $s\in M_n(S)$, we have 
 $$\lim\limits_{t\to \infty} \|(\Theta_t \otimes \id_{M_n})(s)\|  =  \|(\psi \otimes \id_{M_n})(s)\| .$$
 \item[(c)] for $s,s'\in S$, we have 
 $\lim_{t \to \infty} \|\Theta_t(s)\Theta_t(s')\| = \|\psi(s)\psi(s')\|$;
 \item[(d)] if $I$ is $\sigma$-unital, we can also arrange that
$\lim\limits_{t \to \infty} \Theta_t(s)x   =  0$
for all $s \in S$ and $x \in I$. 
\end{itemize} 
Finally, if a completely positive lift for $\psi$ exists,
then the continuous family $\Theta$ as above can be chosen so that, in addition, each map $\Theta_t$ is completely positive.
\end{thm}
\begin{proof}
We begin by introducing some notation.
Let $D$ be a $C^*$-algebra. We denote by $D^\dagger$ the 
``forced'' unitization of $D$, namely $D^\dagger$ is
the (minimal) unitization of $D$ if $D$ is not unital, 
and $D^\dagger=D\oplus \mathbb C$ if $D$ is unital. 
Note that $D$ is an ideal in $D^\dagger$. We denote 
by $\chi_D\colon D^\dagger\to \mathbb C$ the unique state
satisfying $\ker(\chi_D)=D$. This construction is functorial,
and for a linear map $\phi\colon D\to E$ between $C^*$-algebras,
we write $\phi^\dagger\colon D^\dagger \to E^\dagger$ for 
the unique linear extension of $\phi$ which satisfies
$\chi_B\circ \phi^\dagger = \chi_D$. It is easy to see
that $\phi^\dagger$ is a $\ast$-homomorphism if $\phi$ is. Moreover, it is easy to see that the functor $\dagger$ preserves short exact sequences of $C^*$-algebras.
When $\phi$ is a completely positive contractive linear
map, then the same is true for $\phi^\dagger$ by 
Lemma 3.9 in \cite{CE}, or by A.4 on page 266 in \cite{NS}. In particular, if $(D,\delta)$ is a $G$-algebra, then 
$\delta$ extends uniquely to an action $\delta^\dagger\colon 
G\to\Aut(D^\dagger)$ satisfying $\chi_D\circ\alpha_g^\dagger=\chi_D$ for all $g\in G$. 
Finally, if $\phi\colon (D,\delta)\to (E,\epsilon)$ is an
equivariant linear map between $G$-algebras, then 
$\phi^\dagger\colon (D^\dagger,\delta^\dagger)\to 
(E^\dagger,\epsilon^\dagger)$ is also equivariant. 

Applying the above observations to our setting,
we get the extension
\begin{equation*}\label{G-extension1}
\centerline{\begin{xymatrix}{
0 \ar[r] & (I, \gamma) \ar[r]^-\iota & (A^\dagger,\alpha^\dagger) \ar[r]^-{q^\dagger}  & (B^\dagger,\beta^\dagger)  \ar[r] & 0 }
\end{xymatrix}}
\end{equation*}
of $G$-algebras. Moreover,
$\psi^\dagger \colon S^\dagger \to B^\dagger$ is a unital linear completely positive contraction satisfying $\chi_S=\chi_B\circ\psi^\dagger$.
Using that $\chi_S$ is $G$-invariant, 
let $\Theta' = (\Theta'_t)_{t \in [1,\infty)}\colon S^\dagger \to A^\dagger$ be a unital asymptotically $(G,\psi^\dagger)$-equivariant lift of $\psi^\dagger$ to $A^\dagger$
satisfying the conclusion from Theorem~\ref{18-01-21h}. Given $s\in S$, we have 
\[\chi_A \circ \Theta'_t(s) = \chi_B \circ q^\dagger \circ \Theta'_t(s) = \chi_B \circ \psi^\dagger(s) = \chi_S(s) = 0.\]
Hence, the restriction of $\Theta'$ to $S$, which we denote
by $\Theta$, is an asymptotically $(G,\psi)$-equivariant lift of $\psi$ to $A$. Parts (b), (c) and (d) of this theorem follow, respectively, 
from (b), (c) and (d) in Theorem~\ref{18-01-21h}, 
since $S=\ker(\chi_S)$.
Part (a), which is a strengthening of asymptotic $(G,\psi)$-equivariance, follows from (a) in Theorem~\ref{18-01-21h} together
with the fact that
\begin{align*}
\|g \cdot \psi(s) -\psi(h \cdot s')\| &= \|q(g \cdot \Theta_t(s) - \Theta_t(h \cdot s'))\|\\ &
\leq \|g \cdot \Theta_t(s) - \Theta_t(h \cdot s')\| 
\end{align*}
for all $g,h\in G$, all $s,s'\in S$, and all $t\in [1,\infty)$. 
Finally, the last part of this theorem follows from the last part of the statement of Theorem~\ref{18-01-21h}.
\end{proof}

\section{Unital asymptotic sections and amenability}\label{amenabl}

Let $q\colon (A,\alpha)\to (B,\beta)$ be a surjective
$\ast$-homomorphism between $G$-algebras.
An \emph{asymptotically equivariant linear section} for $q$ is an asymptotically linear asymptotic contraction 
$\Theta=(\Theta_t)_{t\in [1,\infty)}\colon B 
\to  A$ such that $q \circ \Theta_t = \id_B$ for all $t$ and such that
$$\lim_{t \to \infty} \Theta_t(g \cdot b)- g\cdot \Theta_t(b)= 0$$ 
for all $g \in G$ and all $b \in B$.

Theorem \ref{04-01-21d} guarantees the existence of an asymptotically equivariant linear section for any extension of $G$-algebras \eqref{G-extensionxxx} with $B$ separable, and in fact one which is also asymptotically completely positive. However, when $A$ is unital so that $B$ is also unital, it is natural to look for an asymptotically equivariant linear section $\Theta$ which is also unital, or at least \emph{asymptotically unital} in the sense that 
$$
\lim_{t \to \infty} \Theta_t(1) = 1.
$$
The section provided by Theorem \ref{04-01-21d} is never
unital, as item (d) shows. Theorem \ref{18-01-21h}, on the other hand, does provide a unital asymptotically equivariant linear section and in fact one which is also asymptotically completely positive, provided there is a $G$-invariant state on $B$. When $G$ is amenable, it is a consequence
of Day's fixed point theorem that every unital $G$-algebra has an invariant state. Therefore, when $G$ is amenable, Theorem \ref{18-01-21h} guarantees the existence of a unital asymptotically equivariant linear section for any 
equivariant surjection $(A,\alpha)\to (B,\beta)$
whenever $A$ is unital and $B$ separable. In fact, said 
theorem provides a unital asymptotically $(G,\id_B)$-equivariant lift of $\id_B$. We show next that this property characterises amenability within the class of second countable locally compact groups.

\begin{thm}\label{05-01-21b} Let $G$ be a second countable locally compact group. The following are equivalent:
\begin{itemize}
\item[(1)] For every compact metrizable $G$-space $X$, the natural extension 
\begin{equation*}\label{14-03-21b}
0\to \mathbb{C} \to C(X)\oplus\mathbb{C} \to C(X)\to 0
\end{equation*} 
admits an asymptotically equivariant linear section which is also aymptotically unital.
\item[(2)]  For every extension \eqref{G-extensionxxx} of $G$-algebras with $B$ separable and $A$ unital there is a unital asymptotically $(G,\id_B)$-equivariant lift of $\id_B$. 
\item[(3)] Every $G$-algebra $(B,\beta)$ with $B$ unital and separable has a $G$-invariant state. 
\item[(4)] $G$ is amenable.
\end{itemize}
\end{thm}
\begin{proof} The implication (4) $\Rightarrow$ (3) follows from the characterization of ame\-na\-bility mentioned above and (3) $\Rightarrow$ (2) follows from Theorem \ref{18-01-21h}. Since (2) $\Rightarrow$ (1) is trivial it suffices to show that (1) $\Rightarrow$ (4). Assume therefore that $G$ has the property stipulated in (1), and 
let $X$ be a compact, metrizable $G$-space. Then $C(X)$ is a $G$-algebra in the natural way and $C(X)$ is separable. Let $p_1\colon C(X)\oplus \mathbb C \to C(X)$ and $p_2\colon C(X)\oplus \mathbb C \to \mathbb C$ be the canonical projections. By assumption, there is an asymptotically equivariant linear section $\Theta\colon C(X) \to C(X) \oplus \mathbb C$ for $p_1$
which is also asymptotically unital. By exchanging $\Theta_t(f)$ with $$
\frac{1}{2}\left(\Theta_t(f) + \Theta_t(f^*)^*\right)
$$ 
for $f\in C(X)$,
we may assume that $\Theta$ is self-adjoint. Let $C_b[1,\infty)$ be the $C^*$-algebra of continuous bounded functions on $[1,\infty)$ and denote by $C_0[1,\infty)$ the ideal in $C_b[1,\infty)$ consisting of the functions vanishing at infinity.
Define $\Phi\colon C(X) \to C_b[1,\infty)$ by $\Phi(f)(t) = p_2 ( \Theta_t(f))$ for $f\in C(X)$, and let 
$$
\pi\colon  C_b[1,\infty) \to  C_b[1,\infty)/C_0[1,\infty)
$$ 
be the quotient map. Then 
$$
\pi \circ \Phi\colon C(X) \to C_b[1,\infty)/C_0([1,\infty)
$$ 
is a linear unital self-adjoint contraction satisfying that 
$$(\pi \circ \Phi)(g \cdot f)= (\pi \circ \Phi)(f)$$ for all $g \in G$ and all $f \in C(X)$. Let $\omega$ be a state of $C_b[1,\infty)/C_0[1,\infty)$. The composition $\omega \circ \pi \circ \Phi$ is then a self-adjoint $G$-invariant linear contraction into $\mathbb C$ which takes $1$ to $1$, and it is therefore a state on $C(X)$. It follows that $X$ has a $G$-invariant Borel probability measure. We have shown that all compact metrizable $G$-spaces have a $G$-invariant Borel probability measure which is one of the many equivalent conditions for amenability of $G$.
\end{proof}

\end{document}